\documentclass[11pt]{amsart}
\usepackage{geometry}
\usepackage{array}
\usepackage{amssymb, amsthm}
\usepackage{ulem}
\usepackage{amsmath}
\usepackage{moreverb}
\usepackage[all]{xy}
\usepackage{graphicx}

\usepackage{url}
\usepackage{hyperref}
\usepackage{enumitem}
\usepackage[backend=bibtex,citestyle=alphabetic,bibstyle=alphabetic,maxbibnames=99]{biblatex}
\usepackage[compress]{cleveref}

\numberwithin{equation}{section}
\renewcommand{\thesubsection}{\arabic{subsection}}
\makeatletter
\renewcommand{\p@subsection}{\thesection.}
\makeatother

\usepackage{zref}
\makeatletter
\zref@newlist{sections}
\zref@newprop{section}{\thesection}
\zref@newprop{subsection}{\thesubsection}
\zref@newprop{secsubsection}{\thesection.\thesubsection}
\zref@addprops{sections}{section,subsection,secsubsection}
\newcommand{\zlabel}[1]{\zref@labelbylist{#1}{sections}}
\newcommand{\zref}[2][section]{\hyperref[#2]{\zref@extractdefault{z:#2}{#1}{\textbf{??}}}}
\makeatother
\AtBeginDocument{%
  \let\oldlabel\label
  \renewcommand{\label}[1]{\oldlabel{#1}\zlabel{z:#1}}}

\newtheorem{thm}{Theorem}[section]
\newtheorem{cor}[thm]{Corollary}
\newtheorem{lem}[thm]{Lemma}
\newtheorem{prop}[thm]{Proposition}

{\theoremstyle{definition}  }
{\newtheorem{rem}[thm]{Remark}}
{\theoremstyle{remark} }

{\theoremstyle{remark} }

 \newcommand{\RR}{\mathbb{R}}
\newcommand{\CC}{\mathbb{C}} 
\newcommand{\ZZ}{\mathbb{Z}}

\newcommand{\schw}[1]{\mathfrak{S}_{#1}}
\newcommand{\disp}[0]{\displaystyle}
\newcommand{\bracepair}[1]{\left\lbrace #1 \right\rbrace}
\newcommand{\anglepair}[1]{\left\langle #1 \right\rangle}
\newcommand{\parenpair}[1]{\left( #1 \right)}
\newcommand{\vertpair}[1]{\left\vert #1 \right\vert}
\newcommand{\Vertpair}[1]{\left\Vert #1 \right\Vert}

\newcommand{\bd}[1]{\partial (#1 )}

\DeclareMathOperator{\tra}{Trace}

\newcommand{\orderasymp}[1]{O \left( #1 \right)}

\renewcommand{\Re}{\operatorname{Re}}
\renewcommand{\Im}{\operatorname{Im}}

\newcommand{\pointmass}[1]{\delta \left( #1 \right) }

\newcommand{\mysteryletterB}{\kappa}
\newcommand{\mysteryletterC}{\kappa}
\newcommand{\mysteryletterD}{\kappa}
\newcommand{\oddcasesig}{\widetilde{\sigma}}
\newcommand{\evencasesig}{\sigma^{\prime}}
\newcommand{\oddcasesigadjust}{\widetilde{\sigma}^{\prime}}
\newcommand{\declare}{\equiv}
\newcommand{\mysteryconstA}{2\pi}

\newcommand{\substituteseq}{\alpha}
\newcommand{\uptologn}{\sim}

\begin{document}
\title[The Spectrum of a Perturbed Harmonic Oscillator Operator]{The spectrum of a Harmonic Oscillator Operator Perturbed by Point Interactions}
\author{Boris S. Mityagin}
\address{231 West 18th Avenue, The Ohio State University, Columbus, OH 43201}
\email{\url{borismit@math.ohio-state.edu}}
%\date{\today}

\begin{abstract}
We consider the operator 

\[
  L = - (d/dx)^2 + x^2 y + w(x) y \quad \text{in } L^2(\mathbb{R}), 
\]
where 
\[
w(x) = s \left[ \delta(x - b) - \delta(x + b) \right], \quad b \neq 0 \, \, \text{real}, \quad s \in \mathbb{C}.
\]
This operator has a discrete spectrum: eventually the eigenvalues are simple and
\begin{equation}
\lambda_n = (2n + 1) + s^2\,  \frac{\kappa(n)}{n} + \rho(n) \label{eq:abstractlam}
\end{equation}
where 
\begin{equation}
\kappa(n) = \frac{1}{\mysteryconstA} \left[(-1)^{n + 1} \sin \left( 2 b \sqrt{2n} \right) - \frac{1}{2} \sin \left( 4 b \sqrt{2n} \right) \right]
\end{equation}
and 
\begin{equation}
\vert \rho(n) \vert \leq C \frac{\log n}{n^{3/2}}. \label{eq:abstracterr}
\end{equation}
If $s = i \gamma$, $\gamma$ real, the number $T(\gamma)$ of non-real eigenvalues is finite, and
\begin{equation}
T(\gamma) \leq \left( C (1 + \vert \gamma \vert) \log (e + \vert \gamma \vert) \right)^2.
\end{equation}
The analogue of \eqref{eq:abstractlam}--\eqref{eq:abstracterr} is given in the case of any two-point interaction perturbation 
\[
w(x) = c_+ \pointmass{x - b} + c_- \pointmass{x + b}, \quad c_+, c_- \in \CC. 
\]
\end{abstract}
\maketitle
\normalem

\section{Introduction}
The operator
\[
L^0 = - \, \frac{d^2}{dx^2} + x^2, \quad x \in \RR^1,
\]
is the one-dimensional harmonic oscillator; this is an unbounded self-adjoint operator acting in $L^2(\RR)$.  As one can see in any introductory book on quantum mechanics, $L^0$ has a discrete spectrum $\Lambda^0 = \bracepair{z_n}_{n = 0}^{\infty}$,
\[ z_n = 2n + 1, \quad n = 0, 1, \dotsc \]
and a compact resolvent
\begin{equation} \label{eq:firstresolve}
R^0(z) = (z - L^0)^{-1} , \quad z \not\in \Lambda^0.
\end{equation}
A normalized orthogonal system of eigenfunctions can be chosen as the Hermite functions
\begin{equation} \label{eq:hefcnintro}
h_n(x) = \left( \pi^{1/2} 2^n n! \right)^{-1/2} e^{-x^2/2} H_n(x), \quad n = 0, 1, \dotsc 
\end{equation}
where
\begin{equation} \label{eq:hepolyintro}
H_n(x) = e^{x^2/2} \left(e^{-x^2/2}\right)^{(n)} 
\end{equation}
are Hermite polynomials.

Spectral analysis of perturbed operators
\begin{equation} \label{eq:elldecompintro}
L = L^0 + W 
\end{equation}
with special $W$, in particular, the point interaction peturbations
\begin{equation} \label{eq:wformintro}
Wf = wf, \quad w(x) = \sum_{j = 1}^J c_j \pointmass{x - b_j}, \quad J \text{ finite}
\end{equation}
was studied in many mathematical and physical papers.

S. Fassari, F. Rinaldi and G. Inglese series of papers \cite{FassIng94}, \cite{FassIng97}, \cite{FassRin} investigate the spectrum of $L \in$\eqref{eq:elldecompintro} when the perturbation
\begin{equation} \label{eq:wspecificintro}
W = - \tau \left( \pointmass{x - b} + \pointmass{x + b} \right), \quad \tau, \beta > 0, 
\end{equation}
i.e., $L^0$ is perturbed by a pair of attractive point interactions of equal strength whose centers are situatied at the same distance from the origin.  In this case the operator $L = L^0 + W$ is self-adjoint; the techniques used are based on Green's function analysis.

D. Haag, H. Cartarius, and G. Wunner \cite{HCW}, motivated by analysis of Bose-Einstein condensates with $\mathcal{PT}$-symmetric loss and gain, focused on the case of non-Hermitian perturbations 
\[
W = i \gamma \left[ \pointmass{x - b} - \pointmass{x + b} \right]. 
\]
Their numerical estimates showed that for small $\gamma$ the spectrum of $L = L^0 + W$ is on the real line $\RR$, and they gave some predictions on the state of decay of the disk radii where the eigenvalues of the operator $L$ are located.  Now we provide a rigorous mathematical analysis of the asymptotics of eigenvalues $\lambda_n = \lambda_n(L^0 + W)$.  

We follow the techniques used in \cite{DM06B}, \cite{DM13Diff}, \cite{AdMipub}, \cite{MiSi}, \cite{Elton03}, \cite{Elton04} and based on careful estimates related to the resolvent representation

\begin{gather}\label{eq:rudefintro}
R  = R^0 + \sum_{j = 1}^{\infty} U_j, \\
U_0 = R^0 , \quad U_k = R^0 W U_{k - 1} = U_{k - 1} W R^0, \quad k \geq 0.
\end{gather}
Moreover, we essentially use the property of perturbations $W \in $\eqref{eq:wformintro} to have such a matrix 
\begin{equation}
w_{jk} = \left\langle W h_j, j_k \right\rangle, \quad j, k = 0, 1, \dotsc
\end{equation}
that for some $\alpha > 0$ there exists $M > 0$ such that 
\begin{equation}\label{eq:wcondintro}
\vert w_{jk} \vert \leq \frac{M}{(1 + j)^{\alpha} (1 + k)^{\alpha}}, \quad j, k = 0, 1, \dotsc , 
\end{equation}
Detailed results on the spectrum and convergence of spectral decompositions of $L = L^0 + W$ for a general $W$ under the condition \eqref{eq:wcondintro} were given by B. Mityagin and P. Siegl \cite{MiSi}.  In the case \eqref{eq:wformintro} of the finite point interaction perturbations $\disp \alpha = \frac{1}{4}$.

\nocite{*}

\section{Preliminaries, Technical Introduction, Review the Results} \label{sec:prelims}

\subsection{}  Our main concern is the harmonic oscillator operator \eqref{eq:elldecompintro} and its special perturbation $W$.  We will focus on this case, although many constructions are very general and could be performed in analysis of other differential operators --- see \cite{DM06B}, \cite{AM12}, \cite{MiSi}, \cite{Elton03}.

Let $L^0$ be an operator in $\ell^2(\ZZ_+)$, 
\begin{equation}
L^0 e_k = z_k e_k, \quad z_k = (2k + 1), \quad k = 0, 1, \dotsc,
\label{eq:ell0}
\end{equation} 
and $W = (w_{jk})_0^{\infty}$ a matrix such that for some $\alpha > 0$ and $C_0 > 0$, 
\begin{equation}
\vert w_{jk} \vert \leq \frac{C_0}{(1 + j)^{\alpha} (1 + k)^{\alpha}} \label{eq:wcond}
\end{equation}
Then the quadratic-form method \cite[see][Section VII.6]{ReedSimon} leads to the definition of the closed operator
\begin{equation}
L = L^0 + W \label{eq:elldef}
\end{equation}
with a dense domain --- see details in \cite{MiSi}.  Let us recall some facts, introduce notations and explain a few elementary but important inequalities.

\subsection{}  To adjust our constructions to the set of eigenvalues of the unperturbed operator \eqref{eq:ell0}, let us define strips
\begin{equation}
\begin{aligned}
H_n & = \bracepair{z \in \CC: \vert \Re z - z_n \vert \leq 1}, \quad n \geq 1\\
H_0 & = \bracepair{z \in \CC: \Re z - z_0 \leq 1}
\end{aligned} \label{eq:hns}
\end{equation}
and the squares 
\begin{equation}
\mathcal{D}_n = \bracepair{z \in H_n: \vert \Re z - z_n \vert \leq \frac{1}{2}, \vert \Im z \vert \leq \frac{1}{2}}, \quad n \geq 0  \label{eq:dn}
\end{equation}
around eigenvalues $ \bracepair{z_n}_{n = 1}^{\infty} = \bracepair{2n + 1}_{n = 0}^{\infty}$ in $H_n$.  

The resolvent 
\begin{equation}
R(z) = (z - L^0 - W)^{-1} \label{eq:res}
\end{equation}
of the operator \eqref{eq:elldef} is well-defined in the right half-plane 
\begin{equation}
\bracepair{z: \Re z \geq 2 N_{*}}  \setminus \bigcup_{k = N_*}^{\infty} \mathcal{D}_k \label{eq:resdom}
\end{equation}
outside of the disks $\mathcal{D}_k$, $k \geq N_*$, if $N_*$ is large enough.  

It follows from the Neumann-Riesz decomposition 
\begin{align}
\begin{aligned}
R & = R^0 + R^0 W R^0 + R^0 W R^0 W R^0 + \dotsb\\
& = R^0 + \sum_{j = 1}^{\infty} U_j,
\end{aligned} \label{eq:rdecomp} \\
\intertext{where}
\begin{aligned}
U_0 &= R^0 = (z - L^0)^{-1},\\
U_j & = R^0 W U_{j - 1} = U_{j - 1} W R^0, \quad j \geq 1.
\end{aligned} \label{eq:udef}
\end{align}
Of course, the convergence of the series should be explained at least in \eqref{eq:resdom}.  This is done in \cite{AM12}, \cite{MiSi}; now I'll remind only the estimates of $N_*$ because it will be important later (see Theorem~\ref{thm:eigendistr}, \eqref{eq:nstarnewdef} and Corollary XXX) in accounting for points of the spectrum $\sigma(L)$ outside of the real line.  

\subsection{}
Define a diagonal operator $K$, 
\begin{equation}
K e_j = \frac{1}{\sqrt{z - z_j}} e_j,\, \,  j = 0, 1, \dotsc, \, \, \Im z \neq 0 \label{eq:kdef}
\end{equation}
with understanding that
\begin{equation}
\sqrt{\xi} = r^{1/2} e^{i \varphi / 2} \text{ if } \xi = r e^{i \varphi}, \quad - \pi < \varphi \leq \pi. \label{eq:sqrootdef}
\end{equation}
Then $K^2 = R^0$, $z \in \CC \setminus \RR$; maybe, we lose analyticity but rough estimates -- when just the absolute values of matrix elements work well -- are good enough.  

Indeed, \eqref{eq:rdecomp}, \eqref{eq:udef} could be rewritten as 
\begin{gather}
R^0 = K^2, \quad U_j = K (K W K)^j K, \label{eq:ujrevise}\\
R = R^0 + \sum_{j = 1}^{\infty} K (K W K)^j K,
\end{gather}
where 
\begin{equation}
(K W K)_{k m} = \frac{1}{\sqrt{z - z_k}} W_{km} \frac{1}{\sqrt{z - z_m}},\, \, k, m = 0, 1, 2, \dotsc, \, \, z \in \CC \setminus \RR. \label{eq:kwkform}
\end{equation}

\begin{lem} \label{lem:hs}
Under the assumptions \eqref{eq:ell0}, and \eqref{eq:wcond}, with $\disp 0 < \alpha < \frac{1}{2}$, if $z \in H_n \setminus \mathcal{D}_n$, then $KWK$ is a Hilbert-Schmidt operator, and
\begin{equation}
\ell \declare \Vert K W K \Vert_{\text{HS}} \leq \frac{C_0 M(\alpha) \log (en)}{n^{2\alpha}} \label{eq:HSnorm}, \quad M(\alpha) \declare 6 + \frac{4/3}{1- 2 \alpha} + \frac{1}{3\alpha}
\end{equation}
\end{lem}

\begin{proof}
If $z \in \bd{\mathcal{D}_n}$, i.e, 
\begin{align}
\begin{aligned}
z = (2n + 1) + \xi + i \eta , & & &\quad \vert \xi \vert = \frac{1}{2}, \, \vert \eta \vert \leq \frac{1}{2},\\
&& \text{ or } & \quad \vert \xi \vert \leq \frac{1}{2}, \, \vert \eta \vert = \frac{1}{2}; \quad \xi, \eta \in \RR,
\end{aligned} \label{eq:bdloc}
\end{align}
then
\begin{equation}
\frac{1}{2} \leq \vert z - z_j \vert \leq 3, \, j = n, n \pm 1, \label{eq:zbd}
\end{equation}
and if $\vert n - j \vert \geq 2$,
\begin{equation}
\frac{3}{2} \vert n - j \vert \leq 2 \vert n - j \vert - 1 \leq \vert z - z_j \vert \leq 2 \vert n - j \vert + 1 \leq \frac{5}{2} \vert n - j \vert.  \label{eq:njest}
\end{equation}
Therefore, by \eqref{eq:wcond}, \eqref{eq:kwkform}, 
\begin{equation}
\ell^2 = \sum_{j, k = 1}^{\infty} \frac{\vert w_{jk} \vert^2}{\vert z - z_j \vert \vert z - z_k \vert} \leq C_0^2 \mu^2 \label{eq:nuexpr}
\end{equation}
with
\begin{equation}
\mu = \sum_{j = 0}^{\infty} \frac{1}{(1 + j)^{2 \alpha} \vert z - z_j \vert}. \label{eq:mudef}
\end{equation}
The sum of three terms for $j = n, n \pm 1$ in \eqref{eq:mudef} by \eqref{eq:zbd} does not exceed
\begin{equation}
3 \cdot \frac{1}{n^{2\alpha}} \cdot 2 = \frac{6}{n^{2\alpha}},  \label{eq:smallvarianceterms}
\end{equation}
and by \eqref{eq:njest}, the remaining part of $\mu$, namely, $ \disp \sum_{j = 0}^{n - 2} + \sum_{j = n + 2}^{\infty} $, by the integral test does not exceed 
\begin{equation}
\begin{aligned}
&\frac{2}{3} \left[ \frac{1}{n} + \frac{1}{n^{2\alpha}} + \int_0^{n - 1} \frac{dx}{(1 + x)^{2\alpha} (n - x)} \right]\\
+ & \frac{2}{3} \left[ \frac{1}{2} \cdot \left( \frac{1}{n + 3} \right)^{2\alpha} + \int_{n + 2}^{\infty} \frac{dy}{(1 + y)^{2\alpha} (y - n)} \right].
\end{aligned} \label{eq:bounda}
\end{equation}
The first integral (after the change of variables $x = n \xi$) is 
\begin{equation}
\begin{aligned}
&\frac{1}{n^{2 \alpha}} \int_0^{1 - (1/n)} \frac{n \, d\xi}{n(1 - \xi) \left( \frac{1}{n} + \xi \right)^{2\alpha}} \leq \\
\leq &  \frac{1}{n^{2 \alpha}} \left[ 2 \int_0^{1/2} \frac{d\xi}{\xi^{2 \alpha}} + 2^{2\alpha} \int_{1/2}^{1 - (1/n)} \frac{d\xi}{1 - \xi} \right] =\\
= & \left( \frac{2}{n} \right)^{2\alpha} \left[ \frac{1}{1-2\alpha} + \log \frac{n}{2} \right].
\end{aligned} \label{eq:bounda2.1}
\end{equation} 
The second integral in \eqref{eq:bounda} is equal to 
\begin{equation}
\begin{aligned}
&\frac{1}{n^{2\alpha}} \int_{1 + (2/n)}^{\infty} \frac{d\eta}{(\eta - 1)\left( \frac{1}{n} + \eta \right)^{2\alpha}} \leq \\
\leq &  \frac{1}{n^{2 \alpha}} \left[  \int_{1 + (2/n)}^{2} \frac{d\eta}{\eta - 1} +  \int_{2}^{\infty} \frac{d\eta}{(\eta - 1)^{2\alpha + 1}} \right] =\\
= & \frac{1}{n^{2\alpha}} \left[ \log \frac{n}{2} + \frac{1}{2\alpha} \right].
\end{aligned} \label{eq:bounda2.2}
\end{equation}
If we collect the inequalities \labelcref{eq:smallvarianceterms,eq:bounda,eq:bounda2.1,eq:bounda2.2}, we get (with $2^{2\alpha} \leq 2$)
\begin{align}
&\begin{aligned}
\mu &\leq \frac{2}{3} \frac{1}{n^{2\alpha}} \left[ 9 + \frac{3}{2} + \frac{2}{1 - 2\alpha} + 3 \log \frac{e n}{2} + \frac{1}{2\alpha}  \right] \leq \\
&\leq \frac{1}{n^{2\alpha}} [M(\alpha) + 2 \log n].
\end{aligned} \label{eq:muest}\\
&\text{ where } M(\alpha) = 6 + \frac{4/3}{1- 2 \alpha} + \frac{1}{3\alpha}. \label{eq:malphdef}
\end{align}
With \eqref{eq:nuexpr} and \eqref{eq:wcond} we come to \eqref{eq:HSnorm}.
\end{proof}

Of course, the constant factors in inequalities \eqref{eq:nuexpr} -- \eqref{eq:muest} are not sharp but we get some idea on their magnitude.  If $\alpha = \disp \frac{1}{4}$ we have 
\begin{gather}
M\left( \frac{1}{4} \right) = 6 + \frac{4}{3} \cdot 2 + \frac{2}{3} < 10, \text{ and} \label{eq:alph1qtr1}\\
\mu \leq \frac{2}{\sqrt{n}} (5 + \log n) \label{eq:alph1qtr2}
\end{gather}
This case is important in analysis of the harmonic oscillator and its perturbations \eqref{eq:wformintro}.  The estimates \labelcref{eq:alph1qtr1,eq:alph1qtr2} will be used later as well.  

\begin{rem} \label{rem:sbds}
Let $\displaystyle s \equiv \sum_{\substack{ j = 0 \\ j \neq n }}^{\infty} \frac{1}{(1 + j)^{\beta}} \cdot \frac{1}{\vert n - j \vert}$.  Then
%Let 
%\begin{equation}
%s \equiv \sum_{\substack{ j = 0 \\ j \neq n }}^{\infty} \frac{1}{(1 + j)^{\beta}} \cdot \frac{1}{\vert n - j \vert}  \label{eq:sdef}.
%\end{equation}
%Then
\begin{subequations}
\renewcommand{\theequation}{\theparentequation.\roman{equation}}
\begin{align}
s & \leq \frac{M(\beta)}{n^{\beta}} \log en, && \text{ if } 0 < \beta \leq 1, \label{eq:sbetasmall}\\
s & \leq \frac{M}{n}, && \text{ if } \beta > 1.  \label{eq:sbetalarge}
\end{align} \label{eq:sconds}
\end{subequations}
\end{rem}

\begin{proof}
The case $\beta = 2 \alpha < 1$ is done in the proof of Lemma~\ref{lem:hs}.  Other cases could be explained in the same way; we omit details.  
\end{proof}

\subsection{} By \eqref{eq:kdef} the operator $K$ is bounded if $z \in H_n \setminus \mathcal{D}_n$ and by \eqref{eq:zbd}, \eqref{eq:njest} its norm
\begin{equation}
\Vert K \Vert \leq \sqrt{2}. \label{eq:knorm}
\end{equation}
Therefore, for $U_j \in $ \eqref{eq:ujrevise} if $j \geq 2$ 
\begin{equation}
\Vert U_j \Vert_1 \leq 2 \Vert K W K \Vert_2^j \leq 2 \nu^j \leq 2 \left[ M(\alpha) \frac{\log e n}{n^{2\alpha}} \right]^j. \label{eq:ujest}
\end{equation}
But we can claim that $U_1$ is a trace-class operator as well, and 
\begin{equation}
\Vert U_1 \Vert_1 = \Vert K(KWK) K \Vert_1 \leq \Vert K \Vert_4 \Vert KWK \Vert_2 \Vert K \Vert_4 \label{eq:u1bds}
\end{equation}
because $K \in \schw{4}$ [or any $\schw{p}$, $p > 2$, as a matter of fact]: just notice that by \eqref{eq:zbd}, \eqref{eq:njest}
\begin{equation}
\begin{aligned}
\Vert K \Vert_4^4& = \sum_{j = 0}^{\infty} \frac{1}{\vert z - z_j \vert^2} \leq \\
&\leq  3/4 + 2 \sum_{k = 2}^{\infty} \left( \frac{2}{3} \right)^2 \cdot \frac{1}{k^2} < 20 < \left( \frac{11}{5} \right)^4,
\end{aligned} \label{eq:k4boundset}
\end{equation}
so
\begin{equation}
\Vert K \Vert_4 \leq \frac{11}{5}; \quad \Vert K \Vert_4^2 \leq 5. \label{eq:k4bound}
\end{equation}   
Therefore we can claim the following.
\begin{prop} \label{prop:normbds}
Under the assumptions \eqref{eq:ell0}, \eqref{eq:wcond}, $0 < \alpha < \frac{1}{2}$, suppose that $N_{*} = N_*(\alpha)$ is chosen in such a way that  
\begin{equation}
M(\alpha) \frac{\log en}{n^{2\alpha}} \leq \frac{1}{2} \quad \text{ for all } \quad n \geq N_*. \label{eq:nstardef}
\end{equation}
Then for $n > N_*(\alpha)$ if $z \in \bd{\mathcal{D}_n}$ all the operators $U_j \in $ \eqref{eq:rdecomp} are of the trace class, their norms satisfy inequalities 
\begin{align}
\Vert U_j \Vert_1 &\leq 2 \left[ M(\alpha) \frac{\log en}{n^{2\alpha}}\right]^j, \quad j \geq 2, \label{eq:ujnorm}\\
\Vert U_1 \Vert_1 & = \Vert R^0 W R^0 \Vert_1 \leq \frac{5 M(\alpha) \log en}{n^{2\alpha}} \label{eq:u1norm}
\end{align}
and the Neumann - Riesz series for the difference of two resolvents
\begin{align}
R - R^0 &= \sum_{j = 1}^{\infty} U_j \label{eq:rdiff}\\
\intertext{converges by the $\schw{1}$-norm and} 
\Vert R - R^0 \Vert_1 &\leq 7 M(\alpha) \frac{\log en}{n^{2\alpha}}\,\label{eq:rdiffnorm}\\
\intertext{and}
\Vertpair{\sum_{j = m}^{\infty} U_j}_1 & \leq 4 \left(  M(\alpha) \frac{\log en}{n^{2\alpha}} \right)^m,\quad m \geq 2. \label{eq:ujsumnorm}
\end{align}
\end{prop}

\begin{proof}
Inequality \eqref{eq:ujnorm} is identical with (proven) line \eqref{eq:ujest}.  \eqref{eq:u1norm} come if we combine \eqref{eq:k4boundset}, \eqref{eq:u1bds}, and \eqref{eq:HSnorm}.  Therefore, for $m \geq 2$, by \eqref{eq:ujnorm} and \eqref{eq:nstardef},
\begin{equation}
\begin{aligned}
\Vertpair{\sum_{j = m}^{\infty} U_j}_1 & \leq 2 \sum_{j = m}^{\infty} \left(  M(\alpha) \frac{\log en}{n^{2\alpha}} \right)^j \leq \\
& \leq 4 \left(  M(\alpha) \frac{\log en}{n^{2\alpha}} \right)^m.
\end{aligned} \label{eq:umanynorm}
\end{equation}
Then by \eqref{eq:u1norm}
\begin{equation}
\begin{aligned}
\Vert R - R^0 \Vert_1 & \leq \Vert U_1 \Vert_1 + \Vertpair{\sum_{j = 2}^{\infty} U_j}_1 \leq \\
& \leq   M(\alpha) \frac{\log en}{n^{2\alpha}}  \cdot \left(5 + 4 M(\alpha) \frac{\log en}{n^{2\alpha}} \right)\\ 
& \leq  7M(\alpha) \frac{\log en}{n^{2\alpha}}.
\end{aligned} \label{eq:rdiffnormest}
\end{equation}
\end{proof}

\section{Deviations of eigenvalues of the harmonic oscillator operator and its perturbations} \label{sec:evalperturb}
\subsection{}Although the constructions and methods of this section are general and applicable to many operators with discrete spectrum and their perturbations, we'll focus later in this section on the case of Harmonic Oscillator operator \eqref{eq:ell0} and its functional representation
\begin{equation}
L^0 y = - y^{\prime \prime} + x^2 y \label{eq:ell0reit}
\end{equation} 
in $L^2(\RR)$.  

The Riesz-Neumann Series \eqref{eq:rdiff}, \labelcref{eq:rdecomp,eq:udef} --- as soon as its convergence in $\schw{1}$ is properly justified --- can be used to evaluate eigenvalues of the operator $L = L^0 + W$.  

Under proper conditions, if $n \geq N_*$, the operator $L$ has the only eigenvalue $\lambda_n$ in $H_n$; moreover, $\lambda_n$ is simple and $\lambda_n \in \mathcal{D}_n$.  Therefore, both of the projections
\begin{align}
P_n^0 & = \frac{1}{2\pi i} \int\limits_{\bd{\mathcal{D}_n}} R^0(z) \, dz = \anglepair{\cdot, h_n} h_n \label{eq:pnbase}\\
\intertext{and}
P_n   & = \frac{1}{2\pi i} \int\limits_{\bd{\mathcal{D}_n}} R(z) \, dz = \anglepair{\cdot, \psi_n} \phi_n \label{eq:pn}
\end{align} 
are of rank $1$.  [In \eqref{eq:pn}, $\phi_n$ is an eigenfunction of $L$ and $\psi_n$ is an eigenfunction of $L^* = L^0 + W^*$, with an eigenvalue $\mu_n = \overline{\lambda}_n$ in $\mathcal{D}_n$.  We will not use this specific information so nothing more is explained now.]

Therefore, 
\begin{align}
\tra P_n^0 = \tra P_n &= 1, \label{eq:trabase}\\
\tra \frac{1}{2\pi i} \int\limits_{\bd{\mathcal{D}_n}} (R(z) - R^0(z)) \, dz &= 0. \label{eq:trafull}
\end{align}
and 
\begin{align}
z_n & = \tra \frac{1}{2\pi i} \int\limits_{\bd{\mathcal{D}_n}} z R^0(z) \, dz = 2n + 1, \label{eq:lambase}\\
\lambda_n & = \tra \frac{1}{2\pi i} \int\limits_{\bd{\mathcal{D}_n}} z R(z) \, dz, \label{eq:lamfull}
\end{align}
So \labelcref{eq:trabase,eq:trafull,eq:lambase,eq:lamfull} imply 
\begin{equation}
\begin{aligned}
\lambda_n - z_n &= \tra \frac{1}{2\pi i} \int\limits_{\bd{\mathcal{D}_n}} (z - z_n) [R(z) - R^0(z)] \, dz\\
& = \sum_{j = 1}^{\infty} T_j(n)
\end{aligned} \label{eq:lamdiff}
\end{equation}
where we put [with $z_n = \lambda_n^0 = 2n + 1$]
\begin{equation}
T_j(n) = T_j(n; W) = \frac{1}{2\pi i} \tra \int\limits_{\bd{\mathcal{D}_n}} (z - z_n) U_j(z) \, dz \label{eq:tjdef}
\end{equation}
Proposition~\ref{prop:normbds} is used in \eqref{eq:lamdiff}, \eqref{eq:tjdef}. \emph{Trace} is a linear bounded functional of norm $1$, on the space $\schw{1}$ of trace-class operators.  It implies the following.
\begin{cor} \label{cor:tjnorm}
Under the assumptions of Proposition~\ref{prop:normbds}, with $n \geq N_*$, we have
\begin{align}
\vert T_j(n) \vert & \leq \left[ M(\alpha) \frac{\log en}{n^{2\alpha}}\right]^j, \quad j \geq 2 \label{eq:tjnorm}
\intertext{and}
\vert T_1(n) \vert & \leq  \frac{9}{4} M(\alpha) \frac{\log en}{n^{2\alpha}} \label{eq:tjnorm.b}
\end{align}
\end{cor}

\begin{proof}
With $\vert \tra A \vert \leq \Vert A \Vert_1$ and 
\begin{equation}
\begin{aligned}
\vert z - z_n \vert \leq \frac{1}{\sqrt{2}},& \quad  z \in \bd{\mathcal{D}_n},\\
\text{length}(\mathcal{D}_n) = 4
\end{aligned} \label{eq:data2}
\end{equation}
rough estimates of integrals \eqref{eq:tjdef} with $j \geq 2$ and $j = 1$ based on \eqref{eq:ujnorm} and \eqref{eq:u1norm} lead to \eqref{eq:tjnorm} and \eqref{eq:tjnorm.b}.
\end{proof}

\begin{cor} \label{cor:lamest}
Under the assumptions of Proposition~\ref{prop:normbds}, with $n \geq N_*$,
\begin{equation}
\lambda_n = (2n + 1) + \sum_{j = 1}^q T_j(n) + r_q(n), \quad q \geq 1, \label{eq:lamest}
\end{equation}
where 
\begin{equation}
\vert r_q(n) \vert \leq 2 \left( M(\alpha) \frac{\log en}{n^{2\alpha}} \right)^{q + 1} \label{eq:rqest}
\end{equation}
\end{cor}

\begin{proof}
The presentation of $\lambda_n$ and the inequality follow from \eqref{eq:lamdiff} and \eqref{eq:ujsumnorm} if we put $m = q + 1$ in \eqref{eq:ujsumnorm} and notice that $2 \sqrt{2} < \pi$ when we multiply the constant factors in inequalities.
\end{proof}
\subsection{Analysis of the function \texorpdfstring{$N_*(\alpha)$}{N sub * alpha}.}
This function is determined by the inequality \eqref{eq:nstardef}.  Later we consider potentials with the coupling coefficient $s$ [see \eqref{eq:pointodd}, \eqref{eq:pointeven}] so it is useful to know the behavior of $X = X_{\beta}(t)$, the solution of the equation 
\begin{equation}
t \frac{\log eX}{X^\beta} = \frac{1}{2}, \quad \beta = 2\alpha, \text{ for large }t. \label{eq:tfromndef}
\end{equation}
Let us rewrite it as 
\begin{align}
\tau \log Y = Y, \, \, & \text{where } Y = (eX)^{\beta}  \label{eq:tfromndef2.1}\\
& \text{and } \tau = \frac{2}{\beta} e^{\beta}t \label{eq:tfromndef2.2}
\end{align}
The equation \eqref{eq:tfromndef2.1} has solutions $1 < y < Y$, such that  
\begin{align}
y(\tau) & = 1 + \frac{1}{\tau} + O \parenpair{\frac{1}{\tau^2}}, &&\tau \to \infty \label{eq:y1def}\\
\intertext{and}
Y(\tau) & \to \infty, && \tau \to \infty. \label{eq:y2def}
\end{align}
\begin{lem}
The solution $Y$ is 
\begin{align}
Y(\tau) &= \tau \log \tau \cdot (1 + r(\tau)) \label{eq:y2behave}\\
\intertext{where}
r(\tau) &= \frac{\log \log \tau}{\log \tau} \left( 1 + o(1) \right) \label{eq:y2r}\\
\intertext{so for any $\delta$ we can find $\tau^*$ such that}
Y(\tau) &\leq \tau \log \tau + \tau(1 + \delta) \log \log \tau, \quad \tau \geq \tau^*. \label{eq:y2bdhigh}\\
\intertext{or $\tau_* < \tau^*$ such that}
Y(\tau) &\leq (1 + \delta) \tau \log \tau, \quad \tau \geq \tau_*. \label{eq:y2bdsub}
\end{align}
\end{lem}

\begin{proof}
If we look for $r \geq 0$, in  \eqref{eq:y2behave}, which solves \eqref{eq:tfromndef2.1} we have:
\begin{align}
\tau \log \tau [1 + r] & = \tau[\log \tau + \log \log \tau + \log(1 + r)] \label{eq:y2proofbegin}\\
\intertext{or}
r = \varphi(r), &\quad \varphi(X) = \xi + \eta \log (1 + X), \quad r > 0 \label{eq:y2proofbegin.2}\\
\intertext{where}
\xi &= \frac{\log \log \tau}{\log \tau}, \quad \eta = \frac{1}{\log \tau} \label{eq:y2proofbegin.3}
\end{align}
For any $\disp 0 < \delta\leq \frac{1}{2}$ we can choose $\tau^*$ such that 
\begin{equation}
0 < \xi \leq \frac{\delta}{2}, \quad 0 < \eta < \frac{\delta}{2} \quad \text{ if }\tau \geq \tau^*.
\end{equation}
Then the function $\varphi$, $\varphi: [0, \delta] \to [0, \delta]$ is a contraction mapping, and \eqref{eq:y2proofbegin.2} has the unique solution 
\begin{equation}
r = r(\tau), \quad 0 < r(\tau) \leq \delta. \quad \label{eq:rconfirm}
\end{equation}
Therefore, 
\begin{equation}
r = \frac{\log \log \tau}{\log \tau} + \frac{\rho}{\log \tau},\quad  0 < \rho \leq \delta. \label{eq:rform}
\end{equation}
This implies \eqref{eq:y2r} with
\begin{equation}
\frac{\rho}{\log \log \tau} = o(1).  
\end{equation}
\end{proof}

\begin{cor} \label{cor:xtbehave}
The solution $X(t)$ of \eqref{eq:tfromndef},  $0 <   \beta \leq 1$, goes to $\infty$ when $t \to \infty$ and 
\begin{equation}
X(t) \leq 2^{1/\beta} \parenpair{t \log \frac{At}{\beta}}^{1/\beta}, \quad A \text{ an absolute constant, }\label{eq:xtest}
\end{equation}
if $t$ is large enough.
\end{cor}
\begin{proof}
If we put \eqref{eq:tfromndef2.2} into \eqref{eq:y2bdhigh} or \eqref{eq:y2bdsub} elementary simplifications give the inequality \eqref{eq:xtest}.
\end{proof}
\subsection{} Inequalities \eqref{eq:nuexpr} and \eqref{eq:ujest} guarantee that we can use the representation \eqref{eq:rdecomp} and eventually ``asymptotics'' \eqref{eq:lamest} if 
\begin{equation}
C_0 M(\alpha) \frac{\log e n}{n^{2\alpha}} \leq \frac{1}{2}, \quad C_0 \in \text{\eqref{eq:wcond}} \label{eq:somebd}
\end{equation}
and $M(\alpha)$ by \eqref{eq:malphdef} is chosen as
\begin{equation}
M(\alpha) = \left[   6 + \frac{4/3}{1- 2 \alpha} + \frac{1}{3\alpha} \right]. \label{eq:malphhelp} 
\end{equation}
Then \eqref{eq:xtest}, with $\beta = 2\alpha < 1$, $t = 2 C_0 M(\alpha)$,
implies that $N_*$ can be chosen as 
\begin{equation}
N_* = N_*(C_0; \alpha) = \left[ 2 C_0 M(\alpha) \log \left(\frac{A}{2\alpha} 2 C_0 M(\alpha)  \right) \right]^{1/(2\alpha)} \label{eq:nstarchoice}
\end{equation}
Now if $\alpha$ is fixed we are interested in the dependence of $N_*$ on $C_0 \in $\eqref{eq:wcond}.
\subsection{} Recall that if $W$ is a multiplier-operator
\begin{equation}
Wf = w(x) f(x), \text{ with } w \in L^p(\RR^1), \quad 1 \leq p < \infty, \label{eq:wmultdef}
\end{equation} \eqref{cor:et2n}
then as we observed and used in \cite{MiSi}
\begin{equation}
w_{jk} = \anglepair{W h_j, h_k} = \int_{-\infty}^{\infty} w(x) h_j(x) h_k(x) \, dx \label{eq:wmatr}
\end{equation}
so by  H\"{o}lder inequality 
\begin{equation}
\vert w_{jk} \vert \leq \Vertpair{w}_p \cdot \Vertpair{h_j}_{2q} \Vertpair{h_k}_{2q}, \quad \frac{1}{p} + \frac{1}{q} = 1. \label{eq:wmatrbd}
\end{equation}
with
\begin{equation}
q > 1, \quad 2q > 2.  \label{eq:qbd}
\end{equation}
But
\begin{equation}
\vertpair{h_k(x)} \leq C k^{-1/12} \label{eq:hsizebd}
\end{equation}
so
\begin{gather}
\begin{aligned}
\int \vert h_k (x) \vert^{2q} \, dx &= \int \vert h_k (x) \vert^2 \vert h_k (x) \vert^{2(q-1)} \, dx\\
& \leq C^{2(q-1)} k^{-(q-1)/6} \int \vert h_k(x) \vert^2 \, dx
\end{aligned} \label{eq:hintbd}\\
\intertext{and}
\Vertpair{h_k}_{2q} \leq C^{1/p} k^{-1/(12p)}, \quad p \geq 1. \label{eq:hnormbd}
\end{gather} 
This means that the matrix $W$ satisfies the condition \eqref{eq:wcond} with $\alpha = \frac{1}{12p}$  This observation was crucial in \cite{MiSi}; it gives a broad class of operators covered by \eqref{eq:wcond} so our claims of this sections are applicable to the operators \eqref{eq:wmultdef}.  

But there are much better estimates of $L^p$ norms of the Hermite functions.

\begin{lem}
As $n \to \infty$,

\begin{subequations}
\begin{align}
\Vertpair{h_n}_r &\sim n^{-\frac{1}{2} \left( \frac{1}{2} - \frac{1}{r} \right)}, \quad 1 \leq r < 4 \label{eq:hefcnsmallr}\\
\Vertpair{h_n}_r &\sim n^{- \frac{1}{8} } \log n, \quad r = 4 \label{eq:hefcnmidr}\\
\Vertpair{h_n}_r &\sim n^{-\frac{1}{6} \left( \frac{1}{r} + \frac{1}{2} \right)}, \quad r >  4 \label{eq:hefcnlarger}
\end{align} \label{eq:hefcnvarr}
\end{subequations}
\end{lem}
See \cite[Lemma 1.5.2]{Thanga} for the sketch of the proof and further explanations of these claims.  

Therefore, \eqref{eq:hnormbd} could be improved.   If $p > 2$ then by \eqref{eq:wmatrbd} $q < 2$, $2q < 4$ so 
\begin{align}
\Vertpair{h_k}_{2q} &\sim k^{-\frac{1}{2} \left( \frac{1}{2} -\frac{1}{2q} \right)} = k^{-\frac{1}{4p}} , && p > 2.\label{eq:hefcnactlargep}\\
\intertext{For $p = 2$ we have $2q = 4$ and}
\Vertpair{h_k}_{4} &\sim k^{-\frac{1}{8} } \log k, && p = 2. \label{eq:hefcnactmidp}\\
\intertext{Finally, if $1 \leq p < 2$ then $2q > 4$ so  }
\Vertpair{h_k}_{2q} &\sim k^{-\frac{1}{6} \left( \frac{1}{2q} + \frac{1}{2} \right)} = k^{-\frac{1}{12} \left(2 - \frac{1}{p} \right)}, && 1 \leq p <2 & \qquad \label{eq:hefcnactsmallp}
\end{align}
All these estimates are used in Proposition~\ref{prop:imagw} , --- see Section~\ref{subsect:genpointpot} below.

Of course, \eqref{eq:hsizebd} shows that $\delta$-potentials
\begin{equation}
w(x) = \sum_{k = 1}^m c_k \delta (x - b_k), \quad m \text{ finite,} \label{eq:wdelt}
\end{equation} 
are good for us as well; in this case, 
\begin{align}
\anglepair{W h_j, h_i} &= \sum_{k = 1}^m c_k h_j(b_k) h_i(b_l) \label{eq:wpairs}
\intertext{and}
\vertpair{W_{ji}} &\leq CM j^{-1/12} i^{-1/12}, \quad i, j \geq 1, \quad M = \sum_{k = 1}^m \vert c_k \vert. \label{eq:wpairbds}
\end{align}
But with more information on asymptotics of Hermite polynomials and Hermite functions we can be accurate in analysis of point-interaction potentials \eqref{eq:wdelt} and spectra of operators
$L^0 + W$, $L^0 \in $ \eqref{eq:ell0reit}, or --- equivalently --- \eqref{eq:ell0}.  This is the main goal of this paper.  Now we go to detailed analysis of these operators.

\section{Two-point interaction potentials} \label{sec:2pt}

\subsection{} \label{subsect:genpointpot} Now we apply general constructions of Sections~\ref{sec:prelims}, \ref{sec:evalperturb} to the case of the two-point interaction potentials 
\begin{equation}
w(x) = c^+ \pointmass{x - b} + c^- \pointmass{x + b}, \quad b > 0 \label{eq:pointmass1}
\end{equation}
and particular cases of an odd potential
\begin{align}
s v^o(x), &\quad s \in \CC \quad \text{ where }&& v^o(x) = \pointmass{x - b} - \pointmass{x + b} \label{eq:pointodd}\\
\intertext{and an even potential}
t v^e(x), &\quad t \in \CC && v^e(x) = \pointmass{x - b} + \pointmass{x + b} \label{eq:pointeven}
\end{align}
--- see \cite{HCW} and \cite{FassRin}.  

Of course, for any \emph{odd} potential \eqref{eq:wmultdef} or \eqref{eq:wdelt}, not just for $v^0 \in$ \eqref{eq:pointodd}, the matrix elements $w_{jk}$ have the property \eqref{eq:orthooddeven}.  Indeed, with parity
\begin{align}
w_{jk} &= \anglepair{w(x) h_j(x),\,  h_k(x)} = \label{eq:woddprop1}\\
&=  \anglepair{w(-x) h_j(-x),\,  h_k(-x)} = -(-1)^{j + k} w_{jk}\label{eq:woddprop2}
\end{align}
so
\begin{equation}
w_{jk} = 0 \text{ if } j + k \text{ even.} \label{eq:orthooddeven} %III.2.2
\end{equation}
If, however, $w$ in \eqref{eq:wmultdef} or \eqref{eq:wdelt} is \emph{even} then we conclude
\begin{equation}
w_{jk} = 0 \text{ if } j + k \text{ odd.} \label{eq:orthoevenodd}
\end{equation}
These observations lead to information on complex eigenvalues of $L = L^0 + W$.
\begin{prop}  \label{prop:imagw}
Let the potential
\begin{align}
w(x) \in L^p, \, \, 1 \leq p < \infty, &\quad \nu = \Vertpair{w}_p, \text{ or} \label{eq:wtype1}\\
w(x) = \sum_{k = 1}^{\infty} c_k \pointmass{x - b_k}, &\quad \nu = \sum \vertpair{c_k} < \infty, \label{eq:wtype2}
\end{align}
be real and odd.  Then the operator
\begin{equation}
L = L^0 + i W = - \frac{d^2}{dx^2} + x^2 + i w \label{eq:newell}
\end{equation}
has at most finitely many non-real eigenvalues, if any, and their number does not exceed $N^*$, where 
\begin{subequations}
\begin{align}
N^* &= D\left (\nu \log (1 + \nu) \right)^{2p}, && p > 2 \label{eq:newnstar.plarge}\\
N^* &= D^* \left (\nu \log^2 (1 + \nu) \right)^{4}, && p = 2 \label{eq:newnstar.pmid}\\
N^* &= D \left (\nu \log (1 + \nu) \right)^{\frac{3}{\left(1 - \frac{1}{2p} \right)}}, && 1 \leq p < 2 \label{eq:newnstar.psmall}\\
N^* &= D_*\left( \nu \log (1 + \nu) \right)^{6}, && \text{ in the case \eqref{eq:wtype2}} \label{eq:newnstar.2}
\end{align} \label{eq:newnstar}
\end{subequations}
and $D^*$, $D_*$ are absolute constants although $D = D(p)$ does not depend on the norm $\nu$.
\end{prop}

\begin{proof}
By \eqref{eq:lamdiff} and estimates in Corollary~\eqref{cor:tjnorm} and in \eqref{eq:nstarchoice} we can use \eqref{eq:lamdiff} to evaluate $\lambda_n$, at least if $p \neq 2$ when we use \eqref{eq:hefcnmidr}.  [Details of this case $p = 2$ are explained in Section~\ref{sec:comment} \end{proof}

\begin{lem} \label{lem:tjodd} %Lemma III.4
If $j$ is odd and $W$ is odd, then for $n \geq N^*$ ($N^*$ as defined in \eqref{eq:newnstar}),
\begin{equation}
T_j(n; W) = 0.  \label{eq:toddzero}
\end{equation}
\end{lem}

\begin{proof}
By \eqref{eq:tjdef}
\begin{align}
T_j(n) = T_j(n; W) = \frac{1}{2\pi i} \tra \int\limits_{\bd{\mathcal{D}_n}} \, \, (z - z_n) U_j(z) \, dz \label{eq:tjdefrep}\\
\intertext{where}
U_j = R^0 W R^0 W \dotsb W R^0 \label{eq:ujdefrep}
\end{align}
with $j$ ``letters'' $W$ and $j + 1$ ``letters'' $R^0$ in this ``word'' $U$.  All these operators are of trace class [see Corollary~\ref{cor:tjnorm}] so $\tra U_j$ is a sum of the diagonal elements $(U_j)_{mm}$ which in turn are sums of matrix elements $\displaystyle \sum_g u(g)$, where $g = \bracepair{g_1, g_2, \dotsc, g_{j-1}} \in \ZZ_+^{j-1}$ and

\begin{equation}
u(g) = \frac{1}{z - z_m} \cdot W(m, g_1) \cdot \frac{1}{z - z_{g_1}} \cdot W(g_1, g_2) \cdot \frac{1}{z - z_{g_2}} \dotsb W(g_{j-1}, m) \cdot \frac{1}{z - z_m} \label{eq:ugdef}
\end{equation}
If we put $g_0 = g_j = m$ we have
\begin{equation}
\sum_{t = 0}^{j-1} (g_{t + 1} - g_t) = g_j - g_0 = 0 \label{eq:gsum}
\end{equation}
The sum of all these differences in \eqref{eq:gsum} is even (zero), so if $j$ is odd, then at least one of these differences, say 
\begin{equation}
g_{\tau + 1} - g_{\tau}, \quad 0 \leq \tau < j,
\end{equation}
is even, and the factor $W(g_{\tau}, g_{\tau + 1})$ in \eqref{eq:ugdef} by \eqref{eq:orthooddeven} is zero.  Therefore, $u(g) = 0$, and this holds for all elements $u(g) \in $ \eqref{eq:ugdef} so $\sum u(g) = 0$.
\end{proof}

\begin{lem} \label{lem:tjeven} %Lemma III.6
If $j$ is even and $W$ is real, then for $n \geq N^*$ ($N^*$ as defined in \eqref{eq:newnstar}),
\begin{equation}
T_j(n) \declare T_j(n; iW) \text{ is real.} \label{eq:tevenreal}
\end{equation}
\end{lem}
\begin{proof}
Now the number $j = 2q$ of $W$-factors in the product \eqref{eq:ugdef} is even for any $g \in \ZZ_+^{(j-1)}$ so 
\begin{equation}
p(g) = \prod_{t = 0}^{j-1} iW(g_t, g_{t + 1}) = (-1)^q \cdot \text{(real number)}
\end{equation}
is real by \eqref{eq:wtype1} or \eqref{eq:wtype2}. 
and
\begin{gather}
u(g) = p(g) \cdot J(g), \text{ where} \label{eq:udecomp} \\ %(III.7.1)
J(g) = \frac{1}{2\pi i} \int\limits_{\bd{\mathcal{D}_n}} F_g(z) \, dz \label{eq:jdef}, \text{ with} \\
F_g(z)  = (z - z_n) \cdot \frac{1}{(z - z_m)^2} \prod_{t = 1}^{j-1} \frac{1}{z - z_{g_t}}  \label{eq:fgzdef}
\end{gather}
For any $g$ this integral is real number [see the next lemma].  Therefore, $T_j(n)$ --- as a sum of (absolutely) convergent series with real terms --- is a real number.
\end{proof}
This completes the proof of Proposition 4.1. 

We will need more specific information about integrals $J(g)\in$ \eqref{eq:jdef}.  The following is true. 

\begin{lem} \label{lem:jgbehave} %Lemma III.7
If $m = n$, and $n \geq N^*$ ($N^*$ as defined in \eqref{eq:newnstar}),
\begin{align}
J(g) = 0 \quad \text{ if at least one } \quad g (\tilde{t}\,  ) = n, \label{eq:jequalzero} \\
\intertext{and}
J(g) = \parenpair{ \prod_{t = 1}^{j-1}2 (n - g_t)}^{-1} \quad \text{otherwise} \label{eq:jequalnonzero}.
\end{align}
If $m \neq n$,
\begin{equation}
J(g) = 0 \,\text{ if }\,\# \tau(g) \neq 2, \, \, \text{where $\tau(g) = \bracepair{t: g_t = n}$,} \label{eq:jnonequalzero} 
\end{equation}
and
\begin{equation}
J(g) = \frac{1}{4(n-m)^2}\parenpair{ \prod_{t \not\in \tau(g)}^{j-1}2 (n - g_{t})}^{-1}  \text{ if } \# \tau(g) = 2. \label{eq:jnonequalnonzero}
\end{equation}
\end{lem}

\begin{proof}
The integrand \eqref{eq:fgzdef} of \eqref{eq:jdef}  could have a pole inside of $\mathcal{D}_n$ only at $z_n = 2n + 1$.  In the cases \eqref{eq:jequalzero} and \eqref{eq:jnonequalzero} the pole's order $\geq 2$ or $F_g(z)$ is analytic on $\overline{\mathcal{D}_n}$, so $J(g) = 0$.  In the cases \eqref{eq:jequalnonzero}, \eqref{eq:jnonequalnonzero} the pole's order is one and $J(g)$ is the residue of $F_g(z)$ at $z_n$.  
\end{proof}

\subsection{An odd potential \texorpdfstring{$v^o$}{v-super-o} \texorpdfstring{in \eqref{eq:pointodd}}{}} \label{subsect:oddpot}  As it is noticed in \eqref{eq:orthooddeven}, 
\begin{gather}
\begin{aligned}
v^0_{jk} &= \anglepair{(\pointmass{x-b} - \pointmass{x + b}) h_j, h_k} \\
 & = [1 - (-1)^{j + k}] a_j a_k
\end{aligned} \label{eq:v0set}\\
= \begin{cases} 0, & \text{ if } j + k \text{ even} \\
2 a_j a_k, & \text{ if } j + k \text{ odd} \end{cases} \label{eq:v0case} \\
\intertext{where}
a_k = h_k(b), \quad k = 0, 1, \dotsc \label{eq:akdef}
\end{gather}
With $b > 0$ fixed, from now on we will use $(a_k)$ as in \eqref{eq:akdef}.  By Lemmas~\ref{lem:tjodd} and \ref{lem:jgbehave} for  $n \geq N^*$
\begin{align}
T_j(n; v^0) \declare T_j(n)& = 0 \, \, \, \, \text{ if } j \text{ odd;} \label{eq:orthooddevenext} \\
\intertext{in particular,}
T_1(n) &= 0 , \quad T_3(n) = 0. \label{eq:orthospec} 
\end{align}
To evaluate $T_2(n)$ we'll sum up (we did it in \eqref{eq:ugdef} in general setting) Cauchy integrals of functions
\begin{equation}
(z - z_n) \cdot \frac{1}{z - z_m} \cdot v_{mk}^0 \cdot \frac{1}{z - z_k} \cdot v_{km}^{0} \cdot \frac{1}{z - z_m} \label{eq:vgdef}
\end{equation}
If $m \neq n$ it is analytic for any $k$ so Cauchy integral is zero.  If $m = n$
\begin{equation}
v_{mk}^0 = 0 \quad \text{ if } n - k \text{ is even.} \label{eq:orthodiff}
\end{equation}
Therefore, by Lemma~\ref{lem:jgbehave}, $j = 2$, with $z_n - z_k = 2(n-k)$, 
\begin{align}
T_2(n; v^0) \declare T_2(n) & = \sum_{\substack{ k = 0 \\ k - n \text{ odd} }}^{\infty} \frac{v_{nk}^0 v_{kn}^0}{z_n - z_k} =  \sum_{\substack{ k = 0 \\ k - n \text{ odd} }}^{\infty}  \frac{2 a_n a_k \cdot 2 a_k a_n }{2(n - k)}  = 2 a_n^2 \oddcasesig(n) \label{eq:t2resolve}\\
\intertext{where}
\oddcasesig(n) &=  \sum_{\substack{ k = 0 \\ n - k \text{ odd} }}^{\infty}  \frac{a_k^2 }{n - k} \label{eq:sigdef}
\end{align}
Technical analysis of the sequence $\oddcasesig(n)$ is done in the next section.  Of course, it is based on asymptotics of Hermite polynomials (or functions), \eqref{eq:hepoly1} -- \eqref{eq:hefcnform}.  It will bring us the proof of the main result of this paper:
\begin{thm} \label{thm:eigendistr} %Thm. III.11
The operator
\[
L = - \frac{d^2}{dx^2} + x^2 + s [ \pointmass{x-b} - \pointmass{x + b}], b > 0, s \in \CC
\]
has a discrete spectrum $\sigma(L)$.  

There exists an absolute constant $D$ such that with 
\begin{equation}
N^* = \left( D \vert s \vert \log e \vert s \vert \right)^2 \label{eq:nstarnewdef}
\end{equation}
all eigenvalues $\lambda_n = \lambda_n(L)$ in the half-plane $\bracepair{z \in \CC: \Re z > N^*}$ are simple, and 
\begin{align}
\lambda_n &= (2n + 1) + s^2 \, \frac{\mysteryletterB(n)}{n} + \widetilde{\rho}(n), \quad \vert \widetilde{\rho}(n) \vert \leq C\frac{\log n}{n^{3/2}}\label{eq:lamasymp}\\
\intertext{where}
\mysteryletterB_n &= \frac{1}{\mysteryconstA} \left[ (-1)^{n + 1} \sin (2 b \sqrt{2n}) - \frac{1}{2} \sin(4b \sqrt{2n}) \right] \label{eq:lamasymp.b}
\end{align} 
\end{thm}

The proof of the theorem is based on the following lemma.

\begin{lem} \label{lem:siginfo} %III.13.
With $\oddcasesig(n) \in$ \eqref{eq:sigdef}
\begin{align}
\oddcasesig(n) & = (-1)^{n + 1}\frac{1}{2} \frac{\sin (2 b \sqrt{2n})}{\sqrt{2n}} + \rho(n), \label{eq:sigdata}\\
\vert \rho(n) \vert &\leq C \frac{\log n}{n} \label{eq:rhobd}
\end{align}
\end{lem}
The technical analysis of this sequence \eqref{eq:sigdef} and its variations is the core of this manuscript.  Its proof is given in the sections which follow.  The final steps to prove Theorem~\ref{thm:eigendistr} are done in Section~\ref{subsect:evenintro}, \eqref{eq:t2repeatresolve} -- \eqref{eq:helperterm3}.  

\subsection{An even potential \texorpdfstring{$v^e \in$ \eqref{eq:pointeven}}{v-super-e}} \label{subsect:evenpot}
Recall \eqref{eq:orthoevenodd}; now 
\begin{align}
v^e_{jk} &= [1 + (-1)^{j + k}] a_j a_k \label{eq:veset}\\
&= \begin{cases} 0, & \text{ if } j + k \text{ odd} \\
2 a_j a_k, & \text{ if } j + k \text{ even} \end{cases} \label{eq:vecase} \\
\end{align}
Therefore, by Lemma~\ref{lem:jgbehave}, $j = 1$, 
\begin{equation}
T_1(n; v^e) \declare T_1(n) = v_{nn}^e = 2 a_n^2 \label{eq:t1even}
\end{equation}
and [compare \labelcref{eq:orthooddevenext,eq:orthospec,eq:vgdef,eq:orthodiff,eq:t2resolve,eq:sigdef}]
\begin{align}
T_2(n; v^e) \declare T_2(n) &= 2 a_n^2 \evencasesig(n),\label{eq:t2even}  \\
\evencasesig(n) & =  \sideset{}{'}\sum_{\substack{k = 0\\ n - k \text{ even}}}^{\infty} \frac{a_k^2}{n - k} \label{eq:sigeven},
\end{align}
where $\sum\nolimits^{'}$ means that $k \neq n$.  

But for the even potential there is no trivial claim $T_3(n) = 0$.  We could make formal references to Lemma~\ref{lem:jgbehave} but let us again look into those terms which form the sum-trace $T_3(n)$.  We integrate functions
\begin{equation}
F = (z - z_n) \cdot \frac{1}{z - z_m} \cdot 2a_m a_k \cdot \frac{1}{z - z_k} \cdot 2 a_k a_{\ell} \cdot \frac{1}{z - z_{\ell}} \cdot 2 a_{\ell} a_m \cdot \frac{1}{z - z_m} \label{eq:fgdef}
\end{equation}  
excluding (by \eqref{eq:vecase}) triples $(m, k, \ell)$ if at least one of the differences $m - k$, $k - \ell$, $\ell - m$ is odd.  

If $m = n$ then we can take only $k, \ell \neq n$, otherwise the order of the pole at $z_n$ would be $\geq 2$ and Cauchy integral \eqref{eq:jdef} be zero.  Then the partial sum of \eqref{eq:tjdefrep} over triples 
\[
\bracepair{(m, k, \ell) \vert m = n, k \neq n, \ell \neq n, k - n, \ell  - n \text{ even}}
\]
would be 
\begin{align}
2 a_n^2 \sum_{\substack{k, \ell \\ n - k, \\ n - \ell \text{ even}}}^{\prime} \frac{a_k^2 a_{\ell}^2}{(n - k)(n - \ell)} = \label{eq:ansumcase1quest}\\
= 2 a_n^2 \left( \evencasesig(n) \right)^2. \label{eq:ansumcase1ans}
\end{align}

If $m \neq n$ Cauchy integral of $F \in$ \eqref{eq:fgdef} is not zero only if $k = \ell = n$, i.e., if we have two (and only two) zeros in the denominator to balance the factor $(z - z_n)$.  This set of triples
\begin{equation}
\bracepair{(m, k, \ell) \vert m \neq n, k= \ell = n, m  - n \text{ even}} \label{eq:mklcase2}
\end{equation}
leads to the subsum in $T_3(n)$ coming from \eqref{eq:fgdef}
\begin{equation}
2 a_n^4 \sideset{}{'}\sum_{\substack{m = 0 \\ m - n \text{ even}}}^{\infty} \frac{a_m^2}{(n - m)^2} = 2 a_n^4 \tau^{\prime}(n) \label{eq:ansumcase2}
\end{equation}
If we combine \eqref{eq:ansumcase1quest} -- \eqref{eq:ansumcase2} we conclude that
\begin{equation}
T_3(n; v^e) = 2 a_n^2 \left[ \sigma^{\prime}(n)^2 + a_n^2 \tau^{\prime}(n) \right] \label{eq:t3even}
\end{equation}
We'll analyze the sequences $\evencasesig$, $\tau^{\prime}$ later as well.  

\section{Inequalities and technical analysis of the term \texorpdfstring{$T_2(n)$}{T2(n)} in the case of an odd two-point \texorpdfstring{$\delta$}{delta}-potential} \label{sec:manyineqs}

First of all, let us recall the asymptotics of Hermite polynomials [see (8.22.8) in Szeg\H{o}, \cite{Szego}].

With $M = 2m + 1$

\begin{equation}
\begin{aligned}
\frac{\Gamma \left( \frac{m}{2} + 1 \right)}{\Gamma (m + 1)} e^{-x^2 / 2} H_m(x) &=& &\cos \left( M^{1/2}x - m \frac{\pi}{2} \right) + \\
&&+& \frac{x^3}{6} M^{-1/2} \sin \left( M^{1/2}x - m \frac{\pi}{2} \right) + O \left( \frac{1}{m} \right).
\end{aligned} \label{eq:hepoly1}
\end{equation}

The normalized Hermite functions $h_m(x)$, $\int h_m^2(x) \, dx = 1,$ are

\begin{equation}
h_m(x) = (m! 2^m \sqrt{\pi} )^{-1/2} H_m(x) e^{-x^2/2} \label{eq:hefcndef}
\end{equation}
so $(m \geq 1)$
\begin{equation}
\begin{aligned}
h_m(x) = \frac{2^{1/4}}{\pi^{1/2}} \frac{1}{m^{1/4}}&  \left[ \cos \left( x \sqrt{2m + 1} - m \frac{\pi}{2} \right) \right.\\
 &+ \left. \frac{x^3}{6} \frac{1}{\sqrt{2m + 1}} \sin \left( x \sqrt{2m + 1} - m \frac{\pi}{2} \right)  + O \left(\frac{1}{m} \right) \right]
\end{aligned} \label{eq:hefcnform}
\end{equation}

Of course, this information is crucial because
\begin{equation}
a_j = h_j(b) \label{eq:haeq}
\end{equation}
and by \eqref{eq:sigdef} 
\begin{equation}
\oddcasesig(n) = \sum_{\substack{k = 0 \\ n - k \text{ odd}}} \frac{a_k^2}{n - k} \label{eq:sigoddrevise}
\end{equation}

\subsection{} \label{subsect:chopandpeel} We will chop and peel this sum by getting ``error'' terms $\rho$'s (with indices) of order $\disp O \left( \frac{\log n}{n}\right).$  After finitely many steps we'll sum up all these error-terms into $\rho(n)$ in \labelcref{eq:sigdata,eq:rhobd}.

First adjustment [to make us flexible to have the denominator $m^{1/4}$ without $m$ being zero] is to take away in \eqref{eq:sigoddrevise} the term with $k = 0$ if $n$ is odd, i.e.
\begin{equation}
\rho_1 = \frac{a_0^2}{n} = \frac{1}{\pi} e^{-b^2} \cdot \frac{1}{n} = O \left( \frac{1}{n}\right). \label{eq:rho1expr}
\end{equation}
Maybe, more drastic is to change $(a_k)$ to the first term $a_k^{\prime}$ in \eqref{eq:hefcnform} so

\begin{align}
a_k^{\prime} {}^2 &= \frac{2^{1/2}}{\pi} \frac{1}{\sqrt{k}} \cos^2 \left( b \sqrt{2k + 1} - k \frac{\pi}{2} \right) \label{eq:aksquareform}\\
& = \frac{1}{\pi} \cdot  \frac{1}{\sqrt{2k}} \left( 1 + (-1)^k \cos 2b \sqrt{2k + 1}\right) \label{eq:aksquareform2}
\end{align}
It will help us but let us explain first that the following is true.

\begin{lem} \label{lem:sumest}%Lemma IV.4
Let $\substituteseq_k$, $\substituteseq_k^{\prime}$ be two sequences such that 

\begin{subequations}
\renewcommand{\theequation}{\theparentequation.\roman{equation}}
\begin{align}
\vert \substituteseq_k \vert, \vert \substituteseq_k^{\prime} \vert \leq \frac{C}{k^{1/4}} \label{eq:alphcond}\\
\delta_k = \substituteseq_k^{\prime} - \substituteseq_k, \quad \vert \delta_k \vert \leq \frac{C}{k^{3/4}} \label{eq:alphdiffcond} 
\end{align} \label{eq:alphakcond}
\end{subequations}
Then for their transforms
\begin{equation}
A(n) = \sum_{\substack{k = 1\ \\ n - k \text{ odd}}}^{\infty} \frac{\substituteseq_k^2}{n - k}, \quad A^{\prime}(n) = \sum_{\substack{k = 1 \\ n - k \text{ odd}}}^{\infty} \frac{\substituteseq_k^{\prime} {}^2}{n - k} \label{eq:ansdef}
\end{equation}
we have
\begin{equation}
\vert A_n \vert, \vertpair{A_n^{\prime}} \leq C \frac{\log n}{\sqrt{n}} \label{eq:ansbds}
\end{equation}
and
\begin{equation}
\vert A_n - A_n^{\prime} \vert \leq C \frac{\log n}{n} \label{eq:andiffsbd}
\end{equation}
\end{lem} 

\begin{proof}
By \eqref{eq:alphcond}

\begin{equation}
\begin{split}
\vert A(n) \vert &\leq \sum_{k \neq n} \frac{\vert \substituteseq_k \vert^2}{\vert n - k \vert} \leq \\
& \leq C^2 \sum \frac{1}{k^{1/2}} \frac{1}{\vertpair{n - k}} \leq \frac{C^{\prime}}{\sqrt{n}} \log (en)
\end{split} \label{eq:anineqs}
\end{equation}
The latter comes from \eqref{eq:sbetasmall} with $\disp \beta = \frac{1}{2}$; of course, the same is true for $A^{\prime}(n)$.  

Next, by \labelcref{eq:alphcond,eq:alphdiffcond}, %\eqref{eq:alphakcond}

\begin{equation}
\begin{split}
\vertpair{\substituteseq_k^2 - \substituteseq_k^{\prime} {}^2} & = \vertpair{\substituteseq_k + \substituteseq_k^{\prime}} \vertpair{\substituteseq_k - \substituteseq_k^{\prime}}\leq \\
& \leq \frac{2C}{k^{1/4}} \cdot \frac{C}{k^{3/4}} = \frac{\widetilde{C}}{k}
\end{split} \label{eq:squarediffs}
\end{equation}
Again, \eqref{eq:sbetasmall}, $\beta = 1$, implies
\begin{equation}
\vertpair{A(n) - A^{\prime}(n)} \leq \frac{C}{n} \log(en). \label{eq:andiffact}
\end{equation}

\end{proof}

Lemma~\ref{lem:sumest} shows that
\begin{equation}
\rho_2(n) = \left\vert \sum_{\substack{k = 0 \\ n - k \text{ odd}}}^{\infty} \frac{a_k^2 - a_k^{\prime} {}^2}{n - k} \right\vert = O \left( \frac{\log n}{n}\right)  \label{eq:rho2bd} 
\end{equation}
so $\rho_2$ is under the restriction \eqref{eq:rhobd}.

In this formula and later on we follow notation \eqref{eq:haeq} and \eqref{eq:sigoddrevise}.

Therefore, we can proceed with a focus on the sequence \eqref{eq:aksquareform2} and the series

\begin{equation}
\oddcasesigadjust(n) = \frac{1}{\pi \sqrt{2}} \sigma(n), \quad \sigma(n) = \sum_{\substack{k = 1 \\ n - k \text{ odd}}}^{\infty}  \frac{1}{\sqrt{k}} \left[ 1 + (-1)^{n + 1} \cos 2b \sqrt{2k + 1}\right] \frac{1}{n - k} \label{eq:sigrenew}
\end{equation}
It is important to point out that the coefficient $(-1)^k$ in \eqref{eq:aksquareform2} dependent on $k$ can be written in \eqref{eq:sigrenew} as $(-1)^{n + 1}$ without dependence on the index $k$ of summation because in this sum only $k \equiv n + 1 \pmod{2}$ are involved.  

Put
\begin{align}
\sigma(n) & = \xi(n) + (-1)^{n + 1} \eta(n) \label{eq:sigrenewdecomp}\\
\intertext{where}
\xi(n) & = \sum_{\substack{k = 1 \\ n - k \text{ odd}}}^{\infty}  \frac{1}{\sqrt{k}} \cdot \frac{1}{n - k} \label{eq:zetarenew}\\
\eta(n) & = \sum_{\substack{k = 1 \\ n - k \text{ odd}}}^{\infty}  \frac{1}{\sqrt{k}} \cdot \frac{\cos 2 b \sqrt{2k + 1}}{n - k}. \label{eq:etarenew}
\end{align} 

\subsection{}  We've already seen (Lemma~\ref{lem:sumest}) that $\disp \vertpair{\xi(n)} = O \parenpair{\frac{\log n}{\sqrt{n}}}$ but such an estimate is not good for $\rho$ in \eqref{eq:rhobd}.  We could hope to get rid of $\log n$ because of the sign-change in the term $\frac{1}{n - k}$.  However, $\xi$ decays much faster than $\disp O \parenpair{\frac{1}{\sqrt{n}}}$.  The following is true.

\begin{lem} \label{lem:zetabds} %IV.8
Let $\xi \in$ \eqref{eq:zetarenew}.  Then 
\begin{equation}
\vertpair{\xi(n)} \leq \frac{C}{n} \label{eq:zetabd}
\end{equation}
\end{lem}

\begin{proof}
The formulas are a little bit different for $n$ even and odd.  Let us write all the details for $n = 2p + 1$ odd fixed; then only even $k = 2m$, $m = 1, 2, \dotsc $, are involved so
\begin{align} \xi(n) &= \xi(2p + 1) = \frac{1}{2 \sqrt{2}} \sum_{m = 1}^{\infty} \frac{1}{\sqrt{m}} \cdot \frac{1}{(p - m) + \frac{1}{2}} \label{eq:zetaexpr0}\\
 &= 2^{-3/2}[S_1 - S_2] \label{eq:zetaexpr}
\end{align}
where
\begin{equation}
\begin{aligned}
S_1 & = \sum_{j = 0}^{p-1} \frac{1}{j + 1/2} \left( \frac{1}{\sqrt{p - j}} - \frac{1}{\sqrt{p + j + 1}} \right) , \\
S_2 & = \sum_{j = p}^{\infty} \frac{1}{j + 1/2} \cdot \frac{1}{\sqrt{p + j + 1}} \, .
\end{aligned} \label{eq:sumforms}
\end{equation}

Put 
\begin{equation}
P = p + \frac{1}{2},\quad t_j = j + \frac{1}{2}, \quad \text{ and} \label{eq:somesymbols}
\end{equation}
\begin{equation}
\varphi(x) = \frac{1}{x} \cdot \frac{1}{\sqrt{P + x}}, \quad x \geq P \label{eq:varphinewdef}
\end{equation}
Then, with the integral test, we compare the sum
\begin{gather}
S_2 = \sum_{j = p}^{\infty} \frac{1}{t_j} \cdot \frac{1}{\sqrt{P + t_j}} = \sum_{j = p}^{\infty} \varphi(t_j) \label{eq:s2def}\\
\intertext{and the integral}
I_2 = \int_{P}^{\infty} \varphi(x) \, dx. \label{eq:i2def}
\end{gather}
Monotonicity of $\varphi(x)$ down implies that 
\begin{equation}
S_2 \geq I_2 \geq S_2 - \varphi(t_p) = S_2 - \frac{1}{\sqrt{2} P^{3/2}} \label{eq:twomono}
\end{equation}
so
\begin{align}
S_2 & = I_2 + \rho_3(n), \label{eq:s2decomp} \\
0 &\leq \rho_3(n) \leq \frac{1}{\sqrt{2} P^{3/2}}; \label{eq:rho3bd}
\end{align}
$\rho_3$ is well under the restriction \eqref{eq:rhobd}.

\eqref{eq:i2def} is evaluated explicitly:
\iffalse
\begin{gather}
\begin{split}
I_2 & = \int_P^{\infty} \frac{dx}{x} \cdot \frac{1}{\sqrt{P + x}} = \frac{1}{P^{1/2}} \int_1^{\infty} \frac{dt}{t} \frac{1}{\sqrt{1 + t}} = \\
& = \frac{1}{P^{1/2}} \int_0^1 \, \frac{dv}{v} \frac{\sqrt{v}}{\sqrt{ 1+ v}} = B \cdot \frac{1}{P^{1/2}} 
\end{split} \label{eq:i2eval}\\
\begin{split}
\text{where } B & = \int_0^1 \frac{dv}{\sqrt{\left(v + \frac{1}{2} \right)^2 - \left( \frac{1}{2} \right)^2}} = \int_1^3 \frac{du}{\sqrt{u^2 - 1}} = \\
& = \int_0^{\mysteryletterC} \frac{d(\cosh s)}{\sinh s}  = \mysteryletterC, \quad \cosh \mysteryletterC = 3, \quad \mysteryletterC > 0
\end{split} \label{eq:bdef}\\
\begin{split}
\text{i.e. } T + \frac{1}{T} = 6, \quad T = e^{\mysteryletterC} > 1,\\
T^2 - 6T + 1 = 0 \quad \text{so}
\end{split} \label{eq:Tdef}\\
\begin{split}
T = 3 + \sqrt{8}, \quad \text{or} \quad \mysteryletterC & = \log (3 + \sqrt{8}) = \\
& = 2 \log(1 + \sqrt{2})
\end{split} \label{eq:constfinal}
\end{gather}
\fi
\begin{equation}
\begin{split}
I_2 & = \int_P^{\infty} \frac{dx}{x} \cdot \frac{1}{\sqrt{P + x}} = \frac{1}{P^{1/2}} \int_1^{\infty} \frac{dt}{t} \frac{1}{\sqrt{1 + t}} = \\
& = \frac{1}{P^{1/2}} \int_0^1 \, \frac{dv}{v} \frac{\sqrt{v}}{\sqrt{ 1+ v}} = B \cdot \frac{1}{P^{1/2}} \\
\text{where } B & = \int_0^1 \frac{dv}{\sqrt{\left(v + \frac{1}{2} \right)^2 - \left( \frac{1}{2} \right)^2}} = \int_1^3 \frac{du}{\sqrt{u^2 - 1}} = \\
& = \int_0^{\mysteryletterC} \frac{d(\cosh s)}{\sinh s}  = \mysteryletterC, \quad \cosh \mysteryletterC = 3, \quad \mysteryletterC > 0\\
\text{i.e. } &T + \frac{1}{T} = 6, \quad T = e^{\mysteryletterC} > 1,\\
&T^2 - 6T + 1 = 0 \quad \text{so}\\
T &= 3 + \sqrt{8}, \quad \text{or}\\
 \quad \mysteryletterC  &= \log (3 + \sqrt{8}) = 2 \log(1 + \sqrt{2})
\end{split} \label{eq:constBfinal}
\end{equation}

[We changed the variables of integration as
\begin{equation*}
\begin{split}
x = Pt, \quad t = \frac{1}{v}, \quad \frac{u}{2} = v + \frac{1}{2}, \quad u = \cosh s .]
\end{split} \label{eq:varchanges}
\end{equation*}
Finally,
\begin{equation}
I_2 = BP^{-1/2}, \quad B = 2 \log (1 + \sqrt{2}). \label{eq:i2final}
\end{equation}

\subsection{}

To analyze $S_1$ let us introduce functions
\begin{gather}
w = \sqrt{P^2 - x^2} \label{eq:wdef}\\
\begin{split}
\psi(x) = \frac{1}{x} \left[ \frac{1}{\sqrt{P - x}} - \frac{1}{\sqrt{P + x}} \right], \\
\frac{1}{2} \leq x \leq P - 1 = p - \frac{1}{2}.
\end{split} \label{eq:psidef}
\end{gather}
Then
\begin{equation}
\begin{split}
\psi(x) &= \frac{1}{x} \cdot \frac{\sqrt{P + x} - \sqrt{P - x}}{w} = \frac{2}{w \left[ \sqrt{P + x} + \sqrt{P - x} \right]} = \\
&= \frac{\sqrt{2}}{w [P + w]^{1/2}}, \quad \text{ monotone increasing,}
\end{split} \label{eq:psimani}
\end{equation}
and the sum 
\begin{equation}
S_1 = \sum_{j = 0}^{p - 1} \psi(t_j), \quad t_j = j + \frac{1}{2} \label{eq:s1def}
\end{equation}
could be compared with the integral
\begin{equation}
I_1 = \int_{1/2}^{p - 1/2} \psi(x) \, dx. \label{eq:i1def}
\end{equation}
Indeed 
\begin{gather}
S_1 \geq I_1 \geq S_1 - \psi(t_{p-1}) \label{eq:s1mono}\\
\psi(t_{p - 1}) = \frac{1}{p - \frac{1}{2}} \left[ 1 - \frac{1}{\sqrt{2p}} \right] \leq \frac{1}{p} \label{eq:psiest1}
\end{gather}
so [compare \labelcref{eq:s2decomp,eq:rho3bd}]
\begin{gather}
S_1 = I_1 + \rho_4(n), \label{eq:s1decomp}\\
0 \leq \rho_4(n) \leq \psi(t_{p-1}) \leq \frac{1}{P}, \label{eq:rho4bd}
\end{gather}
well under the restriction \eqref{eq:rhobd}.

Next, we evaluate $I_1 \in $ \labelcref{eq:i1def,eq:psimani} explicitly, $w = \sqrt{P^2 - x^2}$, $\disp p - \frac{1}{2} = P - 1$,
\begin{align}
\begin{split}
 I_1 &= \int_{1/2}^{p - 1/2} \frac{\sqrt{2} \, dx}{w [ P + w]^{1/2}} =\\
&= \int_{\frac{1}{2} P}^{1 - \frac{1}{P}} \frac{\sqrt{2} P \, dt}{P \sqrt{1 - t^2} \cdot P^{1/2} \left( 1 + \sqrt{1 - t^2} \right)^{1/2}} = \end{split} \label{eq:i1eqs}\\
\begin{split} &= P^{-1/2} \left[ \int_0^1 \frac{ \sqrt{2} \, dt}{\widetilde{w} (1 + \widetilde{w})^{1/2}} - \rho_5^{\prime} - \rho_6^{\prime} \right] \quad \text{where} \quad \widetilde{w}(t) = \sqrt{1-t^2} , \end{split} \label{eq:i1decomp}
\end{align}
\begin{align}
\begin{split}\rho_5^{\prime} &= \int_0^{\frac{1}{2} P} \frac{dt \cdot \sqrt{2}}{\widetilde{w}\left(1 + \widetilde{w}\right)^{1/2}} , \end{split} \label{eq:rho5pdef}\\
 \begin{split}\rho_6^{\prime}&= \int_{1 - \frac{1}{P}}^{1} \frac{dt}{\sqrt{ 1- t}} \cdot \frac{1}{\sqrt{1 + t}} \frac{\sqrt{2}}{\left( 1 + \widetilde{w} \right)^{1/2}} \leq \\
& \leq \sqrt{2} \int_0^{1/P} \frac{d\tau}{\sqrt{\tau}} = 3P^{-1/2}.
\end{split} \label{eq:rho6pdef}
\end{align}
If $P \geq 2$, then $\disp \widetilde{w} \geq \frac{\sqrt{3}}{2}$ and the integrand of $\rho_5^{\prime}$ does not exceed $2$ if $t \in \left[ 0, \frac{1}{4} \right]$ so 
\begin{equation}
\rho_5^{\prime} \leq \frac{1}{P} \label{eq:rho5pbd}
\end{equation}
These inequalities, together with \eqref{eq:i1decomp}, \eqref{eq:rho5pdef}-\eqref{eq:rho5pbd}, \eqref{eq:rho6pdef}, show that
\begin{equation}
I_1 = AP^{-1/2} - \rho_5 - \rho_6 \label{eq:i1truedecomp}
\end{equation}
where
\begin{align}
0 &\leq \rho_5 = P^{-1/2} \rho_5^{\prime}  \leq P^{-3/2}, \label{eq:rho5db}\\
0 &\leq \rho_6 = P^{-1/2} \rho_6^{\prime}  \leq 3P^{-1}, \label{eq:rho6db}
\end{align}
and
\begin{equation}
A = \int_0^1 \frac{\sqrt{2} \, dt}{\widetilde{w}(t) \sqrt{1 + \widetilde{w}(t)}}, \quad \widetilde{w} = \sqrt{1 - t^2}. \label{eq:a2def}
\end{equation}
Let us evaluate $A$; with $t = \sin \varphi$ we have:
\begin{align*}
A & = \int_0^{\pi/2} \frac{\cos \varphi \, d\varphi}{\cos \varphi} \cdot \frac{1}{\left(\frac{1}{2}(1 + \cos \varphi) \right)^{1/2}}= \\
& = \int_0^{\pi/2} \frac{d\varphi}{\cos\frac{\varphi}{2}} \cdot \frac{\cos \frac{\varphi}{2}}{\cos  \frac{\varphi}{2} } = \int_0^{\pi/2} \frac{2 \, d \left( \sin \frac{\varphi}{2} \right)}{1 - \left( \sin \frac{\varphi}{2}\right)^2} \\
& = 2 \int_0^{\sqrt{2}/2} \frac{d\tau}{1 - \tau^2} = \int_0^{1/\sqrt{2}} \left[ \frac{1}{1 - \tau} + \frac{1}{1 + \tau} \right] \, d\tau = \\
& = \left. \log \frac{1 + t}{1 - t} \right\vert_0^{1/\sqrt{2}} = \log (1 + \sqrt{2})^2 = 2 \log(1 + \sqrt{2}), \text{ i.e.,}
\end{align*}
\begin{equation} \label{eq:constArep}
A = 2 \log (1 + \sqrt{2}),
\end{equation}
so by \eqref{eq:i2final}, 
\begin{equation}
A = B, \quad \text{ and } I_1 - I_2 = 0. \label{eq:iscancel}
\end{equation}
By \eqref{eq:zetaexpr}, \eqref{eq:s2decomp}, \eqref{eq:rho3bd}, \eqref{eq:s1decomp}, \eqref{eq:i1decomp}, \eqref{eq:rho5db}, and \eqref{eq:rho6db},

\begin{equation}
\begin{split}
\xi(n) & = 2^{-3/2} [S_1 - S_2] = \\
& = 2^{-3/2} \left[(\rho_4 + AP^{-1/2} - \rho_5 - \rho_6) - (BP^{-1/2} + \rho_3) \right] \\
& = 2^{-3/2} \left[ \rho_4 - \rho_5 - \rho_6 - \rho_3 \right].  
\end{split} \label{eq:zetafinal}
\end{equation}

If $P \geq 4$, with \eqref{eq:rho4bd}, \eqref{eq:rho5db}, \eqref{eq:rho6db}, \eqref{eq:rho3bd} correspondingly, we have: $\disp 2^{-3/2} \leq \frac{2}{5}$ and

\begin{equation}
\begin{split}
\vert \xi(n) \vert &\leq \frac{2}{5} \left[ P^{-1} + P^{-3/2} + 3P^{-1} + P^{-3/2} \right]\\
& \leq \frac{2}{5} P^{-1} \left[ 1 + \frac{1}{2} + 3 + \frac{1}{2} \right] = 2P^{-1}.
\end{split} \label{eq:zetaactbd1}
\end{equation}
With $n = 2p + 1 = 2P$
\begin{equation}
\vert \xi(n) \vert \leq \frac{4}{n} \quad \text{for } n \geq 8. \label{eq:zetaactbd2}
\end{equation}
This proves \eqref{eq:zetabd}, at least for odd $n$.

If $n$ is even, say $n = 2p$, then [see \eqref{eq:zetarenew}], with $k = 2m - 1$,

\begin{align}
\xi(2p) & = \sum_{m = 1}^{\infty} \frac{1}{\sqrt{2m - 1}} \cdot \frac{1}{2(p - m) + 1} \label{eq:zeta2peval} \\
& = \frac{1}{2 \sqrt{2}} \sum_{m = 1}^{\infty} \frac{1}{\sqrt{m - \frac{1}{2}}} \cdot \frac{1}{(p - m) + \frac{1}{2}}. \label{eq:zeta2peval2}
\end{align}
For any $d$, $m > 3\vert d \vert$, for some $\theta$, $\vert \theta \vert < 1$, 
\begin{equation}
\begin{split}
\delta_m &\equiv \left\vert \frac{1}{\sqrt{m - d}} - \frac{1}{\sqrt{m}} \right\vert= \frac{\vert d \vert}{2 (m - d \theta)^{3/2}}\leq\\
& \leq \frac{\vert d \vert }{2 (m - \vert d \vert )^{3/2}} \leq \frac{3^{3/2}}{2^{5/2}} \frac{\vert d \vert }{m^{3/2}} < \frac{ \vert d \vert}{m^{3/2}} \, .
\end{split} \label{eq:deltameval}
\end{equation}
and 
\begin{equation}
\vert \delta_m \vert \leq \vert d \vert m^{-3/2} \quad \text{ if } m > 3 \vert d \vert \label{eq:deltambd}
\end{equation}
By Remark~\ref{rem:sbds}, $\disp \beta = \frac{3}{2}$, $\disp d = \frac{1}{2}$,
\begin{equation}
\begin{split}
\sum_m \vert \delta_m \vert \cdot \frac{1}{\vert p - m + \frac{1}{2}\vert} \leq \frac{C}{p} = \orderasymp{\frac{1}{n}}
\end{split} \label{eq:deltamsum}
\end{equation}
So 
\begin{equation}
\begin{split}
\vert \rho_7(n) \vert = \vert \xi(2p) - \xi(2p + 1) \vert \leq \frac{C}{p} 
\end{split} \label{eq:zetadiffbd}
\end{equation}
and by \eqref{eq:zetaactbd2} and \eqref{eq:zetadiffbd}
\begin{equation}
\vert \xi(2p) \vert \leq \frac{C}{n} \label{eq:zetaactbd3}
\end{equation}
This completes the proof of Lemma~\ref{lem:zetabds}.
\end{proof}

\section{Asymptotics of the term \texorpdfstring{$\eta(n)$}{eta(n)}, Part 1} \label{sec:etabdpt1}
Recall \eqref{eq:etarenew}:
\begin{equation}
\eta(n) = \sum_{\substack{k = 1 \\ n - k \text{ odd}}}^{\infty} \frac{1}{\sqrt{k}} \frac{\cos 2 b \sqrt{2k + 1}}{n - k}. \label{eq:etarepeat}
\end{equation}
\subsection{}  The second term in \eqref{eq:aksquareform2} or \eqref{eq:sigrenew}, \eqref{eq:sigrenewdecomp} leads us to the sequence $\eta(n)$ which --- after Lemma~\ref{lem:zetabds}--- could become a ``leading'' term in the asymptotics of $\evencasesig(n)$ or $\sigma(n) \in$ \eqref{eq:sigrenew} and eventually of $T_2(n)$ in \eqref{eq:t2resolve}, \eqref{eq:t2even}.

One more adjustment along the lines of Remark~\ref{rem:sbds} or inequalities \eqref{eq:deltamsum}, \eqref{eq:zetadiffbd}.  For any $\omega > 0$ and $d$ real, $k > 2 \vert d \vert$,

\begin{equation}
\begin{split}
\delta_k^{\prime} & \equiv \frac{\cos \omega \sqrt{k + d}}{\sqrt{k}} - \frac{\cos \omega \sqrt{k}}{\sqrt{k}} =\\
& = \frac{-\omega}{\sqrt{k}} \cdot \frac{\sin \omega \sqrt{k + \theta_k d}}{2 \sqrt{k + \theta_k d}}, \quad 0 \leq \theta_k \leq 1,
\end{split} \label{eq:deltakpeval}
\end{equation}
so 
\begin{equation}
\vert \delta_k^{\prime} \vert \leq \frac{\omega}{k}, \label{eq:deltakpbd}
\end{equation}
and if in the definition of $\eta(k)$ we'll write $\sqrt{2k}$ instead of $\sqrt{2k + 1}$, i.e., analyze
\begin{equation}
\eta^{\prime}(n) = \sum_{\substack{k = 1 \\ n - k \text{ odd}}}^{\infty} \frac{1}{\sqrt{k}} \cdot \frac{\cos 2 b \sqrt{2k}}{n - k} \label{eq:etapdef}
\end{equation}
the error
\begin{equation}
\rho_8 = \eta - \eta^{\prime}; \quad \vertpair{\rho_8} \leq  \sum_{\substack{k = 1 \\ n - k \text{ odd}}}^{\infty} \vert \delta_k^{\prime} \vert \cdot \frac{1}{\vertpair{n - k}} \label{eq:rho8def}
\end{equation}
by Remark~\ref{rem:sbds} and its variations has an estimate [with \eqref{eq:deltakpbd}]
\begin{equation}
\vertpair{\rho_8(n)} \leq 4 \omega \,\frac{\log n}{n},
\end{equation}
well under restrictions \eqref{eq:rhobd}, so we can analyze $\eta^{\prime} \in $ \eqref{eq:etapdef} instead of $\bracepair{\eta(n)}$ itself.  
\subsection{}\label{subsect:enoddintro} As in Section~\ref{sec:manyineqs}, let us do details when $n$ is odd, i.e., $n = 2p + 1$, so $k = 2m$ and 

\begin{equation}
\begin{split}
\eta^{\prime}(n) &= \sum_{m = 1}^{\infty} \frac{1}{\sqrt{2m}} \cdot \frac{\cos \left(2 b \cdot  2 \sqrt{m}\right)}{2(p-m) + 1} = \\
& = \frac{1}{2\sqrt{2}} \sum_{m = 1}^{\infty} \frac{\cos 4 b \sqrt{m}}{\sqrt{m}} \cdot \frac{1}{(p - m) + \frac{1}{2}}.
\end{split} \label{eq:etapeval}
\end{equation}
As in \labelcref{eq:zetaexpr,eq:sumforms}
\begin{equation}
\eta^{\prime}(n) = 2^{-3/2} [S_1 - S_2]  \label{eq:etapdecomp}
\end{equation}
where
\begin{align}
S_1 & = \sum_{j = 0}^{p-1} \frac{1}{j + \frac{1}{2}} \left[ \frac{\cos 4b \sqrt{p-j}}{\sqrt{p - j}} - \frac{\cos 4b \sqrt{p + j + 1}}{\sqrt{p + j + 1}}\right], \label{eq:es1def}\\
S_2 & = \sum_{j = p}^{\infty} \frac{1}{j + \frac{1}{2}} \frac{\cos 4b \sqrt{p + j + 1}}{\sqrt{p + j + 1}}, \label{eq:es2def}
\end{align}
\begin{prop} \label{prop:es2bds}%Prop V.4
With notations \labelcref{eq:es1def,eq:es2def}
\begin{gather}
\vert S_2 \vert \leq C\frac{1}{p} \label{eq:es2bd}\\
\intertext{and}
S_1  = \pi \, \frac{\sin(4 b \sqrt{p})}{\sqrt{p}} + \rho_9(n) = 2\pi \frac{\sin 2b \sqrt{2n}}{\sqrt{2n}} + \orderasymp{\frac{\log n}{n}}, \label{eq:es1decomp}\\
\vert \rho_9(n) \vert \leq C \frac{\log p}{p} \label{eq:rho9bd}
\end{gather}
\end{prop}
\begin{proof}
Let us start to evaluate $S_2$; this is an easier part.  Like in \labelcref{eq:somesymbols,eq:varphinewdef,eq:s2def,eq:i2def,eq:twomono,eq:s2decomp,eq:rho3bd}
put 
\begin{gather}
\disp \varphi(x) = \frac{1}{x} \frac{\cos r \sqrt{P + x}}{\sqrt{P + x}} = \cos (r \sqrt{P + x})\cdot \psi(x), \quad r = 4b,  \label{eq:evarphidef} \\
\psi(x) = \frac{1}{x \sqrt{P + x}}, \quad P = p + \frac{1}{2} = \frac{1}{2} n.  \label{eq:epsidef}
\end{gather}
The function $\psi$ is monotone decreasing, so $\psi^{\prime}(x)$ is negative and
\begin{equation}
\vert \psi^{\prime}(x) \vert = - \psi^{\prime}(x).  \label{eq:epsisign}
\end{equation}

If 
\begin{equation}
\begin{aligned}
t_j &= j + \frac{1}{2}, \quad j \geq p, \quad \text{ and}\\
I(j) &= [t_j, t_j + 1], \quad \text{then}
\end{aligned} \label{eq:someconsts}
\end{equation} 
\begin{equation}
S_2 = \sum_{j = p}^{\infty} \varphi(t_j) \label{eq:es2sum}
\end{equation}
and with some $\theta_j$, $t_j \leq \theta_j \leq t_{j + 1}$
\begin{equation}
\begin{split}
\Delta_j & \equiv \varphi(t_j) - \int\limits_{I(j)} \varphi(x) \, dx = \varphi (t_j) - \varphi (\theta_j)\\
& = -\int_{t_j}^{\theta_j} \varphi^{\prime}(x) \, dx.
\end{split} \label{eq:edeljdef}
\end{equation}
Therefore,
\begin{equation}
\vertpair{\Delta_j} \leq \int\limits_{I(j)} \vertpair{\varphi^{\prime}(x)} \, dx , \label{eq:edeljbd}
\end{equation}
but by \eqref{eq:evarphidef}
\begin{align}
\varphi^{\prime}(x)& = - \frac{r \sin r \sqrt{P + x}}{2 \sqrt{ P + x}} \psi(x) + \cos r \sqrt{P + x} \cdot \psi^{\prime}(x) \label{eq:evarphider}\\
\intertext{so}
\vertpair{\varphi^{\prime}(x)} &\leq \frac{r}{2} \, \frac{1}{x(P + x)} - \psi^{\prime}(x). \label{eq:evarphiderbd}
\end{align}
Then by \eqref{eq:edeljdef}
\begin{equation}
\begin{split}
\sum_{j = p}^{\infty}\vertpair{\Delta_j} & \leq \sum_{j = p}^{\infty} \int\limits_{I(j)} \vertpair{\varphi^{\prime}(x)} \, dx = \\
& = \int_P^{\infty} \vertpair{\varphi^{\prime}(x)} \, dx \leq \frac{r}{2} \int_P^{\infty} \frac{dx}{x(P + x)} + \psi(P)\\
& = \frac{r}{2} \frac{\log 2}{P} + \frac{1}{P \sqrt{2P}} \leq \frac{2b + 1}{P}.
\end{split} \label{eq:edeljsumbds}
\end{equation}

If 
\begin{equation}
J_2 = \int_P^{\infty} \varphi(x) \, dx \label{eq:j2def}
\end{equation}
we conclude from \eqref{eq:es2sum}, \eqref{eq:edeljdef}, \eqref{eq:edeljsumbds} that
\begin{equation}
S_2 = J_2 + \rho_{10}, \quad \vert \rho_{10} \vert \leq \frac{2b + 1}{P} \label{eq:es2decomp}
\end{equation} 
and instead of the sum $S_2$ we'll analyze the integral $J_2 \in$ \eqref{eq:j2def} with the integrand \eqref{eq:evarphidef}.  A simple absolute value estimates would give the inequality 
\begin{align}
\vert J_2 \vert &\leq \int_P^{\infty} \frac{dx}{x(P + x)^{1/2}} = \mu P^{-1/2} \label{eq:j2bd}\\
\intertext{where}
\mu & = \int_1^{\infty} \frac{d\xi}{\xi(1 + \xi)^{1/2}} \label{eq:emudef}
\end{align}
but this is not good enough.

Recall that by \eqref{eq:j2def}, \eqref{eq:evarphidef}
\begin{equation}
J_2 = \int_{P}^{\infty} \frac{\cos r \sqrt{P + x}}{x \sqrt{P + x}} \, dx;  \label{eq:ej2renew}
\end{equation}
put 
\begin{equation}
\begin{aligned}
P + x &= P(1 + t)^2 \quad \text{so} \\
x &= Pt(2 + t), \quad dx = 2P(1 + t) \, dt 
\end{aligned} \label{eq:evarchange}
\end{equation}
and
\begin{equation}
\begin{split}
J_2 & = \int_{\sqrt{2} - 1}^{\infty} \frac{\cos r P^{1/2} (1 + t)}{Pt(2 + t)P^{1/2}} \cdot \frac{2P (1 + t) \, dt}{(1 + t)} = \\
& = 2P^{-1/2} \int_{\sqrt{2}}^{\infty} \frac{\cos \left(r P^{1/2} \tau \right)\, d\tau}{\tau^2 - 1}
\end{split} \label{eq:ej2eval}
\end{equation}
Now the integrand is absolutely integrable; moreover, we can integrate by parts with 
\begin{equation}
\cos rP^{1/2}\tau = \frac{1}{r P^{1/2}} d \parenpair{\sin r P^{1/2} \tau}, \quad \text{etc.}, \label{eq:ibp1}
\end{equation}
to get
\begin{align}
\rho_{11} = J_2 &= \frac{2 \sin r \sqrt{2p + 1}}{rP} + O \parenpair{\frac{1}{P^{3/2}}} \label{eq:rho11def}\\
\intertext{and}
\vertpair{\rho_{11}} & \leq \frac{3}{r} \, \frac{1}{P} \, , \label{eq:rho11bd}
\end{align}
certainly, under restrictions of \eqref{eq:rhobd}.  

Together with \eqref{eq:es2decomp}
\begin{align}
S_2 & = \rho_{11} + \rho_{10}, \quad \text{ and} \label{eq:es2decompfinal} \\
\vertpair{S_2} & \leq \left[\frac{1}{b} + 2b + 1 \right] \frac{1}{P} \label{eq:es2absbd}
\end{align}
Part \eqref{eq:es2bd} is proven.

\subsection{}  Next, we analyze $S_1 \in $ \eqref{eq:es1def}.  To avoid zero of $(P - x)$ at $x = P$, we consider 
\begin{equation}
S_1^{\prime} = \sum_{j = 0}^{p-2} \frac{1}{j + \frac{1}{2}} \left[ \frac{\cos 4b \sqrt{p-j}}{\sqrt{p - j}} - \frac{\cos 4b \sqrt{p + j + 1}}{\sqrt{p + j + 1}}\right], \label{eq:es1pdef}
\end{equation}
\begin{equation}
\begin{aligned}
S_1 &= S_1^{\prime} + \rho_{12},\\
\rho_{12} & = \bracepair{\text{ the }(p-1) \text{-th term in }S_1}
\end{aligned} \label{eq:es1decomp2}
\end{equation}
and
\begin{equation}
\vertpair{\rho_{12}} \leq \frac{1}{p - \frac{1}{2}} \cdot \left[ \frac{\cos r}{1} + \frac{1}{\sqrt{2p}} \right] \leq \frac{2}{P}, \quad \text{ if } p \geq 5. \label{eq:rho12bd}
\end{equation}

We follow pp. \pageref{eq:wdef}--\pageref{eq:i1decomp}, Section~\ref{sec:manyineqs}, but now the integrand 
\begin{equation}
\begin{split}
g(x) & = \frac{1}{x} \left[ \frac{\cos r \sqrt{ P - x}}{\sqrt{P - x}} - \frac{\cos r \sqrt{P + x}}{\sqrt{ P + x}}\right] \, ,\\
P & = p + \frac{1}{2}, \quad \frac{1}{2} \leq x \leq P - 1 = p - \frac{1}{2},
\end{split} \label{eq:cpxintegrand}
\end{equation}
in the integral
\begin{equation}
J_1 = \int_{1/2}^{p - \frac{1}{2}} g(x) \, dx \quad \left[\omega = \frac{1}{2}, q = p -1 \text{ if to use Lemma~\ref{lem:suminterr}}\right] \label{eq:ej1def}
\end{equation}
is not a monotone function.  Now, as on pages \pageref{eq:evarphidef} -- \pageref{eq:edeljsumbds}, this section, we will use the following.

\begin{lem} \label{lem:suminterr} %Lemma V.10
Let $f$ be a $C^1$ function on the interval $[\omega, \omega + q]$, and
\begin{align}
S &= \sum_{j = 0}^{q-1} f(\omega + j), \label{eq:gensdef}\\
J & = \int_\omega^{\omega + q} f(x) \, dx. \label{eq:genjdef}
\end{align}
Then 
\begin{equation}
\vert S - J \vert \leq \int_\omega^{\omega + q} \vert f^{\prime}(x) \vert \, dx. \label{eq:gensumintdiffbd}
\end{equation}
\end{lem}
This is a well-known claim in elementary Numerical Analysis.
\begin{proof}
As on page~\pageref{eq:someconsts}, \eqref{eq:someconsts} define
\begin{equation}
I(j) = [\omega + j, \omega + j + 1], \quad 0 \leq j \leq q \label{eq:ijdefrenew}
\end{equation}
Then with some $\theta_j$, $0 \leq \theta_j \leq 1$
\begin{equation}
\begin{split}
\Delta_j & \equiv f(\omega + j) - \int\limits_{I(j)} f (x) \, dx = \\
&= f(\omega + j) - f(\omega + j + \theta_j) = -\int_{\omega + j}^{\omega + j + \theta_j} f^{\prime}(x) \, dx
\end{split} \label{eq:gendeltaj}
\end{equation}
so
\begin{equation}
\vertpair{ \Delta_j } \leq \int_{\omega + j}^{\omega + j + 1} \vertpair{f^{\prime}(x)} \, dx, \label{eq:gendeltajbd}
\end{equation}
and
\begin{equation}
\sum_{j = 0}^{q - 1} \vertpair{\Delta_j} \leq \int_\omega^{\omega + q} \vertpair{f^{\prime}(x)} \, dx. \label{eq:gendeltajsum}
\end{equation}
But by the notation \eqref{eq:gendeltaj}
\begin{equation}
J - S = \sum_{j = 0}^{q-1} \Delta_j \label{eq:gensumintcompare}
\end{equation}
so \eqref{eq:gendeltajsum} implies \eqref{eq:gensumintdiffbd}.
\end{proof} %Lemma
This lemma justifies our estimates which follow although we will not necessarily give formal references to Lemma~\ref{lem:suminterr} and its formulas.

\subsection{}  Let us split $g$ of \eqref{eq:cpxintegrand} as
\begin{equation}
g = g_1 + g_2, \quad \text{where} \label{eq:egsplit}
\end{equation}
\begin{equation}
\begin{split}
g_1(x) & = \frac{1}{x} \left[ \frac{1}{\sqrt{P - x}} - \frac{1}{\sqrt{P + x}}\right] \cos r \sqrt{P - x} \\
& = \psi_1(x) \cdot \cos r \sqrt{P - x}
\end{split} \label{eq:eg1def}
\end{equation}
and 
\begin{equation}
\begin{split}
g_2(x) &= \frac{1}{x\sqrt{P + x}} \left[ \cos r \sqrt{P - x} - \cos r \sqrt{P + x} \right] \\
& = \psi_2(x) \cdot \left[ \cos r \sqrt{P - x} - \cos r \sqrt{P + x} \right] 
\end{split} \label{eq:eg2def}
\end{equation}
Notice that as on pages \pageref{eq:wdef}--\pageref{eq:psimani}, \labelcref{eq:wdef,eq:psidef,eq:psimani}
\begin{equation}
\psi_1(x) = \frac{\sqrt{2}}{w(P + w)^{1/2}}, \quad w(x) = \sqrt{P^2 - x^2} \label{eq:psi1def}
\end{equation} 
is a monotone increasing $C^1$-function, and as on p. \pageref{eq:epsidef}, Section~\ref{sec:etabdpt1}, \eqref{eq:epsidef}, 
\begin{equation}
\psi_2(x) = \frac{1}{x \sqrt{P + x}} \quad \text{is monotone decreasing,} \label{eq:psi2def}
\end{equation}
so 
\begin{equation}
\psi_1^{\prime} > 0 \quad \text{and} \quad \psi_2^{\prime} < 0. \label{eq:psiders}
\end{equation}

These properties will help us to deal with integration of $\vertpair{\psi_1^{\prime}(x)}$ or $\vertpair{\psi_2^{\prime}(x)}$ without explicit evaluation of derivatives, just Newton-Leibnitz formula.  

Notice that
\begin{equation}
w^2(P - 1) = \left( p + \frac{1}{2} \right)^2 - \left( p - \frac{1}{2} \right)^2 = 2p \label{eq:wpteval1}
\end{equation}
so
\begin{equation}
0 \leq \psi_1 (P - 1) = \frac{1}{\sqrt{p} \left( P + \sqrt{2p} \right)^{1/2}} \leq \frac{1}{p} \label{eq:psipteval1}
\end{equation}
and 
\begin{equation}
w^2 \left( \frac{1}{2} \right) = P^2 - \frac{1}{4} = \left( p + \frac{1}{2} \right)^2 - \frac{1}{4} = p(p + 1). \label{eq:wpteval2}
\end{equation}
so
\begin{equation}
0 \leq \psi_1 \left( \frac{1}{2} \right) = \frac{\sqrt{2}}{\sqrt{p(p + 1)} \left( P + \sqrt{p(p + 1)}\right)^{1/2}} \leq \frac{1}{p^{3/2}} \label{eq:psipteval2} 
\end{equation}
By \eqref{eq:eg1def}, \eqref{eq:psi1def}
\begin{equation}
g_1^{\prime}(x) = \frac{r \sin r \sqrt{P - x}}{2 \sqrt{P - x}} \psi_1(x) + \psi_1^{\prime}(x) \cos r \sqrt{P - x} \label{eq:eg1der}
\end{equation}
and by \eqref{eq:psiders} and \eqref{eq:psipteval1}
\begin{align}
\begin{aligned}
\int_{1/2}^{P - 1} \vertpair{g_1^{\prime}(x)} \, dx & \leq \frac{r}{2} \int_{1/2}^{P - 1} \frac{\psi_1(x)}{\sqrt{P - x}} + \psi_1(P - 1) \\
& \leq \frac{r}{2} \int_{1/2}^{P - 1} \frac{dx}{(P - x) \sqrt{P + x} \sqrt{P + w}} + \frac{1}{p} \leq 
\end{aligned} \label{eq:g1pinta} \\
\begin{aligned}
2b P^{-1} \int_1^{p} \frac{d\xi}{\xi} + \frac{1}{p} \leq \frac{1}{P} [2b \log p + 1], \quad P \geq 3
\end{aligned} \label{eq:g1pintb}
\end{align}
This error term is $O \parenpair{\frac{\log n}{n}}$ so it is acceptable for \eqref{eq:rhobd}.

For $g_2 \in $ \eqref{eq:eg2def} we have:
\begin{align}
\begin{aligned}
g_2^{\prime}(x) &= &- \psi_2(x) \cdot \frac{r}{2} \left[ \frac{\sin r \sqrt{P - x}}{\sqrt{P - x}}\right] - \psi_2(x) \left[\frac{\sin r \sqrt{P + x}}{\sqrt{P + x}} \right]\\
&&+ \psi_2^{\prime}(x) \left[ \cos r \sqrt{P - x} - \cos r \sqrt{P + x} \right]
\end{aligned} \label{eq:eg2der}\\
\begin{aligned}
& \equiv v_1(x) + v_2(x) + v_3(x)
\end{aligned} \label{eq:eg2derdecomp}
\end{align}  

Then by \eqref{eq:psi2def}, $r = 4b$,
\begin{align}
\begin{split}
\int_{1/2}^{P - 1} \vert v_1(x) \vert \, dx & \leq 2b \int_{1/2}^{P - 1} \frac{dx}{x \sqrt{P + x} \sqrt{P - x}} \leq \\
& \leq 2b P^{-1/2} \int_{\frac{1}{2P}}^{1} \frac{P \, dt}{PtP^{1/2} \sqrt{ 1  - t}} \leq \\
& \leq 2b P^{-1} \left( \frac{3}{2} \int_{\frac{1}{2P}}^{1/2} \frac{dt}{t} + 2 \int_{1/2}^1 \frac{dt}{\sqrt{1 - t}}\right) \leq 
\end{split} \label{eq:v1inta}\\
\begin{split}
& \leq \frac{3b}{P} \left[ \log P + 1 \right]
\end{split} \label{eq:v1intb} 
\end{align}  
Next,
\begin{align}
\begin{aligned}
\int_{1/2}^{P - 1} \vert v_2(x) \vert \, dx & \leq \frac{r}{2} \int_{1/2}^{P - 1} \frac{dx}{x(P + x)} = \\
& = 2b P^{-1} \int_{1/2}^{P - 1} \left[ \frac{1}{x}- \frac{1}{P + x} \right] \, dx \leq 
\end{aligned} \label{eq:v2inta}\\
\begin{aligned}
& \leq \frac{2b}{P} \log 2P
\end{aligned} \label{eq:v2intb}
\end{align}  
Finally,

\begin{align}
&\vertpair{\cos r \sqrt{P - x} - \cos r \sqrt{P + x}} \leq \label{eq:sometermsbda}\\
\leq &\, \,  r (\sqrt{P + x} - \sqrt{P - x}) = \frac{8bx}{\sqrt{P + x} + \sqrt{P - x}}  \leq \frac{8bx}{\sqrt{P + x}}. \label{eq:sometermsbdb}
\end{align}  
and
\begin{equation}
\begin{split}
- \psi_2^{\prime}(x) &= \frac{1}{x^2} \, \frac{1}{\sqrt{P + x}} + \frac{1}{2x(P + x)^{3/2}}, \\
\intertext{so}
\int_{1/2}^{P - 1} \vertpair{v_3(x)} & \leq 4 \sqrt{2} b \int_{1/2}^{P - 1} \left[ \frac{1}{x \sqrt{P + x}} \cdot \frac{1}{(P + w)^{1/2}} + \frac{1}{(P + x)^{3/2}} \cdot \frac{1}{(P + w)^{1/2}} \right] \, dx
\end{split} \label{eq:v3inta}
\end{equation}
and
\begin{align}
\int_{1/2}^{P - 1} \vert v_3(x) \vert \, dx & \leq 8b \int_{1/2}^{P - 1} \left[ \frac{1}{x(P + x)} + \frac{1}{2(P + x)^2}\right] \, dx \label{eq:v3intb} \\
& \leq \frac{8b}{P} \left( \log 2P + \frac{1}{2} \right) \label{eq:v3intc}
\end{align}

With our notations \labelcref{eq:eg2der,eq:eg2derdecomp}, if we collect inequalities \labelcref{eq:v1intb,eq:v2intb,eq:v3intb} we have:
\begin{equation}
\int_{1/2}^{P - 1} \vertpair{g_2^{\prime}(x)} \, dx \leq \frac{13b}{P} \log(4P) . \label{eq:eg2derbd}
\end{equation}
Together with \eqref{eq:g1pintb} and \eqref{eq:egsplit} this implies
\begin{equation}
\int_{1/2}^{P - 1} \vertpair{g^{\prime}(x)} \, dx \leq \frac{15b}{P} \log (4P) \label{eq:egintbd}
\end{equation}
By Lemma~\ref{lem:suminterr},
\begin{align}
S_1^{\prime} & = J_1 + \rho_{14}, \label{eq:es1pdecomp} \\
\vertpair{\rho_{14}} & \leq \frac{15b}{P} \log(4P), \label{eq:rho14bd}
\end{align}
well under restrictions \eqref{eq:rhobd}.

All of the above make $J_1$ a front-runner in the race to be ``a leading term'' of asymptotics of $\eta(n)$ and $\sigma(n)$.  

\subsection{} \label{subsect:f1bds} To analyze $J_1 \in$ \eqref{eq:ej1def} we split the integrand $g \in$ \eqref{eq:cpxintegrand} on the interval $\disp \left[ \frac{1}{2}, p - \frac{1}{2} \right]$ just as it is written there, i.e., 
\begin{equation}
g(x) = f_1(x) - f_2(x) \label{eq:egnewdecomp}
\end{equation}
where
\begin{align}
f_1(x) & = \frac{1}{x} \frac{\cos r \sqrt{P - x}}{\sqrt{P - x}}\, , \label{eq:ef1def} \\
f_2(x) & = \frac{1}{x} \frac{\cos r \sqrt{P + x}}{\sqrt{P + x}} \, . \label{eq:ef2def} 
\end{align}

There are no zeros in the denominators; all functions $g$, $f_1$, $f_2$ are $C^{\infty}$ on $\left[ \frac{1}{2}, p - \frac{1}{2} \right]$ and for a while we can manipulate $f_1$, $f_2$ and their integrals separately.  When the arguments of $\mathit{cos}$ would be linear functions it is easier to compare two components and balance two integrals
\begin{equation}
Q_i = \int_{1/2}^{p - \frac{1}{2}} f_i(x) \, dx, \quad i = 1, 2. \label{eq:qidef}
\end{equation} 
To evaluate $Q_1$ put
\begin{align}
P - x &= P(1-\tau)^2 \quad \text{ so} \label{eq:q1changevara} \\
x & = P\tau(2 - \tau), \quad dx = 2P(1 - \tau) \, d \tau. \label{eq:q1changevarb}
\end{align}
The lower bound $\disp x = \frac{1}{2}$ corresponds to 
\begin{gather}
\tau_* = 1 - \left( 1 - \frac{1}{2P}\right)^{1/2} = \frac{1}{4P} + \epsilon, \label{eq:taulowdef}\\
0 \leq \epsilon \leq \frac{1}{(2P)^2} \label{eq:rhotaubd}
\end{gather}
and the upper bound $x = P - 1$ leads to 
\begin{equation}
\tau^* = 1 - \frac{1}{\sqrt{P}}. \label{eq:tauhighdef}
\end{equation}
Therefore,
\begin{align}
Q_1 & = \int_{\tau_*}^{\tau^*} \frac{\cos \parenpair{r P^{1/2} (1 - \tau)} 2P(1 - \tau) \, d\tau}{P\tau(2 - \tau)P^{1/2}(1 - \tau)} \label{eq:q1evala} \\
& = 2P^{-1/2} \int_{\tau_*}^{\tau^*} \frac{\cos \parenpair{r P^{1/2} (1 - \tau)}}{\tau} \, \frac{d\tau}{2 - \tau} \label{eq:q1evalb} 
\end{align} 

For unification, let us change $\tau_*$ to $\frac{1}{4P}$ and $\tau^*$ to $1$, so 
\begin{equation}
\begin{split}
\int_{\tau_*}^{\tau^*} = \int_{\frac{1}{4P}}^{1} - \int_{\tau^*}^{1} - \int_{\frac{1}{4P}}^{\frac{1}{4P} + \epsilon} \equiv \int_{\frac{1}{4P}}^{1} + \rho_{15}^{\prime} + \rho_{16}^{\prime}, \quad \tau_* \geq \frac{1}{4P}
\end{split} \label{eq:q1bdschange}
\end{equation}
On $\disp \left[ \frac{1}{4P}, \frac{1}{4P} + \epsilon \right]$ the absolute value of integrand in \eqref{eq:q1evalb} does not exceed $4P$ but --- see \eqref{eq:rhotaubd} --- the length of this interval $\leq \frac{1}{4P^2}$; therefore,
\begin{align}
\vertpair{\rho_{15}^{\prime}} &\leq \frac{1}{P} \quad \text{and} \label{eq:rho15pbd}\\
\rho_{15} &= 2P^{-1/2} \rho_{15}^{\prime}, \quad \vertpair{\rho_{15}} \leq 2P^{-3/2}. \label{eq:rho15bd}
\end{align}
On $[\tau^*, 1]$ the function - integrand has absolute values $\leq 2$ if $P \geq 4$ (see \eqref{eq:tauhighdef}) so 
\begin{align}
\vertpair{\rho_{16}^{\prime}} & \leq 2(1 - \tau^*) = \frac{2}{\sqrt{P}} \label{eq:rho16pbd} \\
\intertext{and}
\rho_{16} &= 2P^{-1/2} \rho_{16}^{\prime}, \quad \vertpair{\rho_{16}}. \leq 4P^{-1}. \label{eq:rho16bd}
\end{align}
Therefore, by \eqref{eq:q1evalb}, \eqref{eq:q1bdschange}
\begin{equation}
Q_1 = P^{-1/2} \int_{\frac{1}{4P}}^{1} \frac{\cos \parenpair{r P^{1/2} (1 - \tau)} }{\tau} \frac{2d\tau}{2 - \tau} + \rho_{15} + \rho_{16}. \label{eq:q1decompfinal}
\end{equation}
One more adjustment: change factor $\disp \frac{2}{2 - \tau}$ to $1$:
\begin{equation}
\frac{2}{2 - \tau} = 1+ \frac{\tau}{(2-\tau)}, \label{eq:denomchange}
\end{equation}
so the integral in \eqref{eq:q1decompfinal} is equal to
\begin{equation}
 \int_{\frac{1}{4P}}^{1} \frac{\cos \parenpair{r P^{1/2} (1 - \tau)} }{\tau} \, d\tau + \rho_{17}^{\prime} \label{eq:q1denomrevise}
\end{equation}
where
\begin{equation}
\begin{split}
\rho_{17}^{\prime} &= \int_{\frac{1}{4P}}^{1} \frac{\cos \parenpair{r P^{1/2} (1 - \tau)} }{2 - \tau} \, d\tau = \\
& = \left[ \int_0^1 - \int_0^{\frac{1}{4P}} \right]
\end{split} \label{eq:rho17pdef}
\end{equation}
The absolute value of the integrand in $\rho_{17}^{\prime}$ is bounded by $1$ so the second integral 
\begin{equation}
\vertpair{\int_0^{\frac{1}{4P}}\frac{\cos \parenpair{r P^{1/2} (1 - \tau)} }{2 - \tau} \, d\tau } \leq \frac{1}{4P} \label{eq:rho17pint2bd}
\end{equation}
and 
\begin{align}
\rho_{17}^{\prime} &= \int_0^1 \frac{\cos \left( r P^{1/2} y \right)}{1 + y} \, dy + \rho_{18}^{\prime}, \label{eq:rho17pdecomp} \\
\vertpair{\rho_{18}^{\prime}} & \leq \frac{1}{4} P^{-1} \label{eq:rho18pbd}
\end{align}
Riemann lemma tells us that for any $C^1$-function $\mysteryletterD$ on the interval $[A, B]$
\begin{equation}
\vertpair{\int_A^B e^{iTx} \mysteryletterD(x) \, dx} \leq \frac{1}{T} \left[ \vertpair{\mysteryletterD(A)} + \vertpair{\mysteryletterD(B)} + \int_A^B \vertpair{\mysteryletterD^{\prime}(x)} \, dx\right] \label{eq:riemleb}
\end{equation}
Therefore, the integral in $\rho_{17}^{\prime} \in $ \eqref{eq:rho17pdecomp} does not exceed
\begin{equation}
\frac{1}{rP^{1/2}} \left[ 1 + \frac{1}{2} + \frac{1}{2} \right] = \frac{1}{2b} \cdot P^{-1/2}, \label{eq:rho17pint1bd}
\end{equation}
and the adjustment \eqref{eq:q1denomrevise}, i.e., 
\begin{align}
\rho_{17} &= P^{-1/2} \rho_{17}^{\prime}, \quad \text{satisfies} \label{eq:rho17adjust} \\
\vertpair{\rho_{17}} &\leq \frac{1}{2b} \cdot P^{-1} + \vertpair{\rho_{18}^{\prime}} \label{eq:rho17bd} 
\end{align}

If we collect \eqref{eq:rho15bd}, \eqref{eq:rho16bd}, \eqref{eq:rho17bd}, and \eqref{eq:rho18pbd}, we get 
\begin{equation}
\begin{split}
\delta \equiv \sum_{i = 15}^{18} \vertpair{\rho_i} &\leq P^{-1} \left[ 2P^{-1/2} + 4 + \frac{1}{2b} + \frac{1}{4} P^{-1/2} \right] \\
& \leq \parenpair{7 + \frac{1}{2b}} P^{-1} 
\end{split} \label{eq:q1deltabd}
\end{equation}
and with this error, instead of $Q_1 \in$ \eqref{eq:q1decompfinal} we can consider 
\begin{equation}
Q_1^{\prime} = P^{-1/2} \int_{\frac{1}{4P}}^1 \frac{\cos r P^{1/2} (1 - \tau)}{\tau} \, d\tau \, . \label{eq:q1pdef} 
\end{equation}
Indeed, by \eqref{eq:q1deltabd}
\begin{equation}
\vertpair{Q_1 - Q_1^{\prime}} \leq \parenpair{7 + \frac{1}{2b}} P^{-1}. \label{eq:q1q1pdiff}
\end{equation}
\subsection{} Now we evaluate
\[
Q_2 = \int_{1/2}^{P - 1} f_2(x) \, dx, \quad f_2 \in \text{ \eqref{eq:ef2def}}.
\]
Many details are the same as in Section~\ref{subsect:f1bds} or (pp. \pageref{eq:j2def}--\pageref{eq:es2absbd}) Section~\ref{subsect:enoddintro}.  We'll make explanations short but all the formulas are written explicitly.  As in \eqref{eq:ej2renew}--\eqref{eq:evarchange} put
\begin{equation}
\begin{aligned}
P + x &= P(1 + t)^2 \quad \text{so} \\
x &= Pt(2 + t), \quad dx = 2P(1 + t) \, dt 
\end{aligned} \label{eq:evarchangerenew}
\end{equation}
and
\begin{equation}
Q_2 = \int_{t_*}^{t^*} \frac{\cos \left( r P^{1/2} (1 + t) \right) \, 2P(1 +t) \, dt}{Pt(2 + t) P^{1/2}(1 + t)} \label{eq:q2def}
\end{equation}
where
\begin{equation}
\begin{split}
t_* =\left( 1 + \frac{1}{2P}\right)^{1/2} - 1 = \frac{1}{4P} - \epsilon^{\prime}, \\
0 \leq \epsilon^{\prime} \leq \frac{1}{4P^2}; 
\end{split} \label{eq:tlowdef}
\end{equation}
\begin{equation}
\begin{split}
t^* =\left( 2 - \frac{1}{P}\right)^{1/2} - 1 = \sqrt{2} - 1 - \epsilon^{\prime\prime},\\
0 \leq \epsilon^{\prime \prime} \leq \frac{1}{2P}. 
\end{split} \label{eq:thighdef}
\end{equation}

With
\begin{equation}
Q_2 = P^{-1/2} \int_{t_*}^{t^*} \frac{\cos \left( r P^{1/2} (1 + t) \right) }{t(2 + t)} \, 2\, dt \label{eq:q2eval}
\end{equation}
If we change $t_*$ to $\frac{1}{4P}$ we are losing the interval $\disp \left[ t_*, \frac{1}{4P} \right]$ but its length is $\disp \leq \frac{1}{4P^2}$ [see \eqref{eq:tlowdef}] and the integrand's absolute value does not exceed $4P$ so the total ``error'' $\rho_{19}$ is 
\begin{equation}
\vertpair{\rho_{19}} \leq 4P \cdot \frac{1}{4P^2} \cdot 2P^{-1/2} = 2P^{-3/2} \label{eq:rho19bd} 
\end{equation}
The change of $t^*$ to $\sqrt{2} - 1$ gives the ``error'' $\rho_{20}$, and 
\begin{equation}
\vertpair{\rho_{20}} \leq \frac{1}{2P} P^{-1/2}= \frac{1}{2}P^{-3/2} \label{eq:rho20bd} 
\end{equation}

Now we change the factor $\disp \frac{2}{2 + t}$ to $1$ as in \eqref{eq:denomchange}, \eqref{eq:q1denomrevise}; 
\begin{equation}
\frac{2}{2 + t} = 1 - \frac{t}{2 + t}, \label{eq:denomchange2}
\end{equation}
and the difference-integral
\begin{equation}
\rho_{21} = P^{-1/2} \int_{\frac{1}{4P}}^{\sqrt{2} - 1} \frac{\cos \parenpair{r P^{1/2} (1 - t)} }{(2 + t)} \, dt  \label{eq:rho21def}
\end{equation}
so by \eqref{eq:riemleb} we have
\begin{align}
\vertpair{\rho_{21}} \leq \frac{2}{r P^{1/2}} P^{-1/2} = \frac{1}{2b} \cdot \frac{1}{P} \label{eq:rho21bd}
\end{align}
As in \eqref{eq:q1deltabd}
\begin{equation}
\begin{split}
\delta^{\prime} \equiv \sum_{i = 19}^{21} \vertpair{\rho_i} &\leq \frac{1}{P} \left[ 2P^{-1/2} + \frac{1}{2} P^{-1/2}  + \frac{1}{2b} \right] \\
& \leq \frac{1}{P}\left[ 2 + \frac{1}{2b} \right], \quad P \geq 4  
\end{split} \label{eq:q2deltabd}
\end{equation}
and
\begin{equation}
\vertpair{Q_2 - Q_2^{\prime}} \leq \left[ 2 + \frac{1}{2b} \right] \cdot \frac{1}{P} \label{eq:q2q2pdiff}
\end{equation}
where
\begin{equation}
Q_2^{\prime} =  P^{-1/2} \int_{\frac{1}{4P}}^{\sqrt{2} - 1}\frac{\cos r P^{1/2} (1 - t) }{1 + t} \, dt \label{eq:q2pdef}
\end{equation}
\subsection{} The two previous subsections [see \eqref{eq:ej1def}, \eqref{eq:q1pdef}, \eqref{eq:q2pdef}] explained that
\begin{equation}
J_1 = Q_1^{\prime} - Q_2^{\prime} + O \parenpair{\frac{\log P}{P}} \quad , \label{eq:ej1finaldecomp}
\end{equation}
 even with more accurate estimates of the error term $\disp O \parenpair{\frac{\log P}{P}}$.  Before the concluding claims about the $J_1$'s (i.e., $\sigma(n)$) asymptotics let us bring an integration for $Q_1^{\prime}$ and $Q_2^{\prime}$ to the same interval, say, $\disp \left[ \frac{1}{4P}, \frac{1}{4} \right]$ with understanding that $\frac{1}{4} < 1$ and $\frac{1}{4} < \sqrt{2} - 1$, so a piece of $Q_1^{\prime}$ 
 
\begin{align}
\rho_{22}^{\prime} &= \int_{1/4}^1 \frac{\cos r P^{1/2}(1 - \tau)}{\tau} \, d\tau \label{eq:rhoq1bc}\\
\intertext{and a piece}
\rho_{23}^{\prime} & = \int_{1/4}^{\sqrt{2} - 1} \frac{\cos r P^{1/2}(1 + \tau)}{\tau} \, d\tau \label{eq:rhoq2bc}
\end{align}
could be estimated by Riemann Lemma \eqref{eq:riemleb}.  We have
\begin{equation}
\begin{split}
\rho_{22} = P^{-1/2} \rho_{22}^{\prime}, \quad &\vertpair{\rho_{22}} \leq \frac{4 + 4}{rP} = \frac{1}{b} \cdot \frac{2}{P}
\end{split} \label{eq:rhoq1bcbd}  
\end{equation}
and
\begin{equation}
\begin{split}
\rho_{23} = P^{-1/2} \rho_{23}^{\prime}, \quad &\vertpair{\rho_{23}} \leq \frac{1}{b} \cdot \frac{2}{P}
\end{split} \label{eq:rhoq2bcbd}  
\end{equation}
Therefore
\begin{equation}
J_1 = \Delta + \rho_{22} - \rho_{23} \label{eq:ej1anotherdecomp}
\end{equation}
where 
\begin{align}
\Delta &= P^{-1/2} \delta, \label{eq:anotherDelta}\\
\delta & \declare \int_{\frac{1}{4P}}^{1/4} \frac{1}{t} \left[\cos r P^{1/2}( 1 - t) -  \cos r P^{1/2}( 1 + t) \right] \, dt \label{eq:anotherdelta}
\end{align}
An elementary identity
\[
\cos u - \cos v = 2 \sin \parenpair{\frac{u + v}{2}} \sin \parenpair{\frac{v - u}{2}}
\]
implies that
\begin{equation}
\delta = 2 \parenpair{\sin r P^{1/2}} \cdot \int_{\frac{1}{4P}}^{1/4} \frac{\sin r P^{1/2} t}{t} \, dt \label{eq:anotherdeltarewrite}
\end{equation}
But
\begin{align}
\rho_{24}^{\prime} &= \vertpair{\int_0^{\frac{1}{4P}} \frac{\sin r P^{1/2} t}{t} \, dt } \leq \frac{1}{4P} \cdot r P^{1/2} \quad \text{and} \label{eq:rho24pbd} \\
\rho_{24} &= 2P^{-1/2} \rho_{24}^{\prime}, \quad \vertpair{\rho_{24}} \leq 2 b P^{-1} \label{eq:rho24bd}
\end{align}
Next, 
\begin{align}
\int_0^{1/4} \frac{\sin r P^{1/2} t}{t} \, dt & = \int_0^{\frac{1}{4} r P^{1/2}} \frac{\sin x}{x} \, dx\\
&= \frac{\pi}{2} - \rho_{25}^{\prime}, \label{eq:sinint}
\end{align}
with
\begin{equation}
\rho_{25}^{\prime} = \int_{\frac{1}{4} r P^{1/2}}^{\infty} \frac{\sin x}{x} \, dx \label{eq:rho25pdef}
\end{equation}
where by \eqref{eq:riemleb}
\[
\vertpair{\rho_{25}^{\prime}} \leq \frac{2}{bP^{1/2}}
\]
and
\begin{equation}
\rho_{25} = 2P^{-1/2} \rho_{25}^{\prime}, \quad \vertpair{\rho_{25}} \leq \frac{4}{b} P^{-1}. \label{eq:rho25bd}
\end{equation}
By \eqref{eq:ej1anotherdecomp}, \eqref{eq:anotherDelta}, \eqref{eq:anotherdeltarewrite}, \eqref{eq:sinint} we conclude 

\begin{equation}
\begin{split}
\vertpair{ J_1 - \pi  \, \frac{\sin r P^{1/2}}{P^{1/2}} } & \leq \sum_{i = 22}^{25} \vertpair{\rho_i} \leq\\
& \leq \left( b + \frac{10}{b} \right) \cdot P^{-1}
\end{split} \label{eq:ej1finalest}
\end{equation}
and with \eqref{eq:ej1finaldecomp} and \eqref{eq:es1pdecomp}, \eqref{eq:rho14bd}, \eqref{eq:rho11def}, \eqref{eq:etapdecomp} the statement \eqref{eq:es1decomp} is proven.
\end{proof} %main prop

\section{Asymptotics of the term \texorpdfstring{$\eta(n)$}{eta(n)} -- Part 2} \label{sec:etabdpt2}
\subsection{} \label{subsect:oddcomplete} In Section~\ref{sec:etabdpt1} from the start --- see \eqref{eq:etapeval} --- we did everything for odd $n$.  But if $n = 2p$ is even we have to take in the sum \eqref{eq:etarepeat} only odd $k = 2m - 1$, $m = 1, 2, \dotsc$ so 
\begin{equation}
\begin{split}
\eta(n) = \eta(2p) & = \sum_{m = 1}^{\infty} \frac{1}{\sqrt{2m + 1}} \cdot \frac{\cos 2b \sqrt{4m - 1}}{2(p - m) + 1}\\
& = \frac{1}{2\sqrt{2}} \sum_{m = 1}^{\infty} \frac{\cos 4b \sqrt{m - \frac{1}{4}}}{\sqrt{m - \frac{1}{2}}} \cdot \frac{1}{(p - m) + \frac{1}{2}} .
\end{split} \label{eq:etaeven}
\end{equation}
At the same time [see \eqref{eq:etarepeat}]
\begin{equation}
\eta(2p + 1) = \frac{1}{2\sqrt{2}} \sum_{m = 1}^{\infty} \frac{\cos 4b \sqrt{m + \frac{1}{4}}}{\sqrt{m}} \cdot \frac{1}{(p - m) + \frac{1}{2}} . \label{eq:etaodd}
\end{equation}

We've already done such estimates at least twice but --- to avoid any confusion --- let us notice (with $r = 4b$): the following.
\begin{equation}\label{eq:exprprime}
\left( \frac{\cos r \sqrt{x + \frac{1}{4}}}{\sqrt{x}} \right)^{\prime} = - \, \frac{r \sin r \sqrt{x + \frac{1}{4}}}{2 \sqrt{x + \frac{1}{4}} \cdot \sqrt{x}} - \frac{\cos r \sqrt{x + \frac{1}{4}}}{2x^{3/2}}
\end{equation}
so if $m \geq 2$
\begin{equation}\label{eq:eomegamdef} 
\begin{aligned}
\omega_m &= \left\vert \frac{\cos r \sqrt{m - \frac{1}{4}}}{\sqrt{m- \frac{1}{2}}} - \frac{\cos r \sqrt{m + \frac{1}{4}}}{\sqrt{m}} \right\vert \leq \\
& \leq \frac{1}{2} \cdot \frac{1}{2} \left[ \frac{r}{m - \frac{1}{2}} + \frac{1}{\left( m - \frac{1}{2} \right)^{3/2}}\right] \leq \\
& \leq \frac{1}{4} \left[ \frac{\frac{4}{3} r}{m}  + \left( \frac{4}{3} \right)^2 \frac{1}{m^{3/2}}\right] \leq \\
&\leq \frac{1}{m} (2b + 1).
\end{aligned}
\end{equation}
These inequalities by Remark~\ref{rem:sbds}, $\beta = 1$, imply [compare \eqref{eq:andiffact}] that
\begin{equation}
\vertpair{\eta(2p) - \eta(2p + 1)} \leq \frac{1}{2 \sqrt{2}} \sum_{m = 1}^{\infty} \omega_m \cdot \frac{1}{\vert p - m \vert + \frac{1}{2}}\leq C (2b + 1) \frac{\log(ep)}{p}. \label{eq:etaoddevendiff}
\end{equation}
\subsection{}  Now we are ready to complete the proof of the following.
\begin{prop} \label{prop:sigmatwiddlebounds} %Prop. VI.3.
\begin{equation}
\oddcasesig(n) = \frac{(-1)^{n + 1}}{2} \frac{\sin 2 b \sqrt{2n}}{\sqrt{2n}} + O \parenpair{\frac{\log n}{n}} \label{eq:finalsigmatwiddle}
\end{equation}
\end{prop}
\begin{proof} %line 1924
With $u \uptologn v$ meaning that $\disp \vert u - v \vert = \orderasymp{\frac{\log n}{n}}$, we had by \eqref{eq:etapdecomp} \eqref{eq:rho11def}, \eqref{eq:es1pdecomp}--\eqref{eq:rho14bd}, \eqref{eq:ej1finaldecomp}, and \eqref{eq:ej1finalest}
\begin{equation} \label{eq:mainchain}
\begin{aligned}
\widetilde{\sigma}(n) &\uptologn \frac{1}{\pi \sqrt{2}} \sigma(n) \uptologn \frac{(-1)^{n + 1}}{\pi \sqrt{2}} \eta(n) \uptologn\\
& \uptologn \frac{(-1)^{n + 1}}{\pi \sqrt{2}} \eta^{\prime}(n) \uptologn \frac{(-1)^{n + 1}}{\pi \sqrt{2}} \cdot 2^{-3/2} S_1 \uptologn\\
& \uptologn \frac{(-1)^{n + 1}}{4\pi} J_1 \uptologn \frac{(-1)^{n + 1}}{4\pi} \cdot 2 \pi \cdot \frac{\sin 2b \sqrt{2n}}{\sqrt{2n}} \\
& \uptologn \frac{(-1)^{n + 1}}{2} \cdot \frac{\sin 2b \sqrt{2n}}{\sqrt{2n}}
\end{aligned}
\end{equation}
\end{proof}

\subsection{} \label{subsect:evenintro} In Section~\ref{sec:2pt}, \eqref{eq:t2resolve}, \eqref{eq:sigdef} we've seen that in the case of the odd potential $v^o(x) = \pointmass{x - b} - \pointmass{x + b}$
\begin{align}
T_1(n) = T_3(n) & = 0, \quad \text{and}\\
T_2(n) &= 2 a_n^2 \oddcasesig(n) \label{eq:t2orepeat} 
\end{align}
---see \eqref{eq:t2resolve} and Propositions~\ref{prop:es2bds} and \ref{prop:sigmatwiddlebounds}.  Now we know $\oddcasesig(n)$ --- see Lemma~\ref{lem:siginfo} or Prop.~\ref{prop:sigmatwiddlebounds}.  Of course, $(a_n)$ come from (8.22.8) in \cite{Szego} --- see \eqref{eq:hefcndef}, \eqref{eq:hefcnform} --- so   
\begin{equation}
\begin{split}
a_n^2 &= \frac{1}{\pi} \cdot \frac{1}{\sqrt{2n}} \left[ 1 + (-1)^n \cos 2b \sqrt{2n} \right] + \orderasymp{\frac{1}{n}}.
\end{split} \label{eq:an2repeat}
\end{equation}
Recall that by \eqref{eq:finalsigmatwiddle}, \eqref{eq:es1decomp},
\begin{equation}
\oddcasesig(n) = \frac{1}{2} (-1)^{n + 1} \frac{\sin 2 b \sqrt{2n}}{\sqrt{2n}} + \orderasymp{ \frac{\log n}{n} } \label{eq:anothersigmaodd}
\end{equation}

Therefore, by \eqref{eq:t2orepeat}
\begin{align}
T_2(n) &= \frac{\mysteryletterB(n)}{n} + O \parenpair{\frac{\log n}{n^{3/2}}} \label{eq:t2repeatresolve}\\
\intertext{where}
\mysteryletterB(n) &= \frac{1}{\mysteryconstA} \left[(-1)^{n + 1} \sin 2 b \sqrt{2n} - \frac{1}{2} \sin 4b \sqrt{2n} \right]. \label{eq:helperterm3}
\end{align}
This proves \eqref{eq:lamasymp.b}, i.e., the second part of Theorem~\ref{thm:eigendistr}.

\subsection{} So far we were dealing with an odd potential $v^o \in$ \eqref{eq:pointodd}.  But the technical hurdles we've overcome help to answer the questions on the asymptotics of eigenvalues in the case of an even perturbation $v^e$, --- or $tv^e$, $t \in \CC$, --- see \eqref{eq:pointeven}.

In this case (see Section~\ref{subsect:evenpot}) 
\begin{align}
T_1(n) &= 2 a_n^2 \quad \text{ by \eqref{eq:sigeven}} \label{eq:t1erepeat}\\
T_2(n) &= 2 a_n^2 \evencasesig(n), \label{eq:t2erepeat}\\
\evencasesig(n) & = \sum_{\substack{k = 0 \\ n - k \text{ even} \\ k \neq n}}^{\infty} \frac{a_k^2}{n-k} \label{eq:sigmaprepeat}
\end{align}
and
\begin{align}
T_3(n) & = 2 a_n^2 [ \evencasesig(n) + a_n^2 \tau^{\prime}(n)] \label{eq:t3erepeat}\\
\intertext{where}
\tau^{\prime}(n) &= \sideset{}{'}\sum_{\substack{m = 0 \\ m - n \text{ even}}}^{\infty} \frac{a_m^2}{(n - m)^2} \label{eq:tauprepeat}
\end{align}
For a while we'll consider only two traces $T_1$, $T_2$ to get asymptotics \labelcref{eq:lamest,eq:rqest},
\begin{equation}
\lambda_n = (2n + 1) + T_1(n) + T_2(n) + O \left( \parenpair{\frac{\log n}{n}}^3\right) \label{eq:lamemidest}
\end{equation}
in explicit form.  
\subsection{} Of course, $T_1(n) \in$ \eqref{eq:t1erepeat} comes from  (8.22.8) in \cite{Szego} --- see \labelcref{eq:hepoly1,eq:hefcnform}, or 
\begin{equation}
\begin{split}
a_n = h_n(b)  = \frac{2^{1/4}}{\pi^{1/2}} \, \frac{1}{n^{1/4}} & \left[ \cos \parenpair{b \sqrt{2n + 1} - n \frac{\pi}{2}} \right. \\
&\left. + \frac{b^3}{6} \, \frac{1}{\sqrt{2n + 1}} \sin \parenpair{b \sqrt{2n + 1} - n \frac{\pi}{2}} + \orderasymp{\frac{1}{n}} \right].
\end{split} \label{eq:anformmid}
\end{equation}
But we want to collect in \eqref{eq:lamemidest} the coefficients for $\disp \frac{1}{\sqrt{n}}$ and $\disp \frac{1}{n}$ [ sending ``smaller'' terms into $\rho = \disp O \parenpair{\parenpair{\frac{\log n }{n}}^3}$ ] so now we need carefully to watch factors 
\begin{equation}
\sqrt{2n} \text{ and } \sqrt{2n + 1} = \sqrt{2n} + \frac{1}{2 \sqrt{2n}} + \orderasymp{\frac{1}{n^{3/2}}} \label{eq:cares}
\end{equation}
or functions' values at the points $\sqrt{2n}, \sqrt{2n + 1}$.  It was not necessary in analysis of $v^o$ in Section~\ref{sec:manyineqs}--\ref{sec:etabdpt1} because the accuracy of $\disp \rho = O \parenpair{\frac{\log n}{n}}$ in Lemma~\ref{lem:siginfo} did not depend on the term $\disp \frac{1}{2\sqrt{2n}}$ in \eqref{eq:cares} and we could be indifferent to any divergence between $\sqrt{2n}$ and $\sqrt{2n + 1}$.  However, now we need to write 
\begin{equation}
\begin{split}
\cos  2b \sqrt{2n + 1} &= \cos \parenpair{2b \sqrt{2n} + \frac{b}{\sqrt{2n}}} + \orderasymp{\frac{1}{n^{3/2}}} \\
& = \cos 2b \sqrt{2n} - \frac{b}{\sqrt{2n}} \sin 2b \sqrt{2n} + O\left( \frac{1}{n} \right)
\end{split} \label{eq:coschange}
\end{equation}
so in the sequence
\begin{equation}
\begin{split}
a_n^2 = & \frac{1}{\pi \sqrt{2n}} \left[ 1 (-1)^n \cos 2b \sqrt{2n} \right.\\
& \left. - \frac{(-1)^n}{\sqrt{2n}} b \left( 1- \frac{b^2}{3} \right) \sin 2b \sqrt{2n} + \orderasymp{\frac{1}{n}} \right]\\
= &\frac{1}{\pi} \, \frac{1}{\sqrt{2n}}  \left[ 1 + (-1)^{n} \cos 2b \sqrt{2n} \right] + \\
& + \frac{(-1)^{n + 1}}{2\pi} \cdot \frac{1}{n}\, b  \left(1 - \frac{b^2}{3} \right) \sin \parenpair{2b \sqrt{2n}}   + \orderasymp{\frac{1}{n^{3/2}}},
\end{split} \label{eq:a2eeval}
\end{equation}
an additive term with the coefficient $\disp \frac{1}{n}$ is important.

\subsection{} \label{subsect:tauomega0} The term $T_2 \in$ \eqref{eq:t2erepeat} has a factor $a_n^2$ so when we evaluate $\evencasesig(n)$ we can ignore --- as we did in Sections~\ref{sec:manyineqs}, \ref{sec:etabdpt1} ---- the second term of order $\disp \frac{1}{n}$ in \eqref{eq:a2eeval}.  In particular, if $n = 2p$ is even we evaluate
\begin{equation}
\evencasesig(n) = \sum_{\substack{m = 0 \\ m \neq p}}^{\infty} \frac{a_{2m}^2}{2(p-m)}. \label{eq:sigmapt2e}
\end{equation}
We can follow all the constructions and analogues of inequalities of Sections~\ref{sec:manyineqs}, \ref{sec:etabdpt1} but as in Section~\ref{subsect:oddcomplete} we can avoid a new 30 page writing but observe that this presumably new sum adjusted by Lemma~\ref{lem:sumest} looks as   
\begin{equation}
\frac{1}{2\pi} \sum_{\substack{m = 1 \\ m \neq p}}^{\infty} \frac{1}{\sqrt{m}} \, \left[ 1 + \cos 4b \sqrt{m} \right] \cdot \frac{1}{p - m} \label{eq:notsonewsum}
\end{equation}
(compare \eqref{eq:sigrenew} or \eqref{eq:etarenew}), i.e., we need to give good estimates for
\begin{align}
\tau_0(p) = \sum_{\substack{m = 1 \\ m \neq p}}^{\infty} \frac{1}{\sqrt{m}} \, \frac{1}{p - m} \label{eq:tau0def}
\intertext{and}
\tau_1(p) = \sum_{\substack{m = 1 \\ m \neq p}}^{\infty} \frac{\cos 4b \sqrt{m}}{\sqrt{m}} \, \frac{1}{p - m}. \label{eq:tau1def}
\end{align}
In Sections~\ref{sec:manyineqs} and \ref{sec:etabdpt1} we've analyzed the sequences
\begin{align}
\omega_0(p) &= \sum_{m = 1}^{\infty} \frac{1}{\sqrt{m}} \, \frac{1}{p - m + \frac{1}{2}} \quad \text{(see \eqref{eq:sumforms})}
\intertext{and}
\omega_1(p) &= \sum_{m = 1}^{\infty} \frac{\cos 4b \sqrt{m}}{\sqrt{m}} \, \frac{1}{p - m + \frac{1}{2}} \quad \text{(see \eqref{eq:etapeval})}
\end{align}
The following is true.
\begin{lem} \label{lem:tauomegadiff} %VI.12
\begin{subequations}
\begin{align}
\vert \tau_0(p) - \omega_0(p) \vert & \leq C \frac{\log p}{p^{3/2}} \label{eq:tauomega0diff}\\
\vert \tau_1(p) - \omega_1(p) \vert & \leq C \frac{\log p}{p} \label{eq:tauomega1diff}
\end{align} \label{eq:tauomegadiffs}
\end{subequations}
\end{lem}
\begin{proof}[Part (a)]  As in \eqref{eq:zetaexpr}, \eqref{eq:sumforms} we rewrite $\tau_0(p)$ as 
\begin{equation}
\begin{split}
&\sum_{j = 1}^{p - 1} \frac{1}{j} \left[ \frac{1}{\sqrt{p-j}} - \frac{1}{\sqrt{p + j}} \right] - \sum_{j = p}^{\infty} \frac{1}{j} \, \frac{1}{\sqrt{p + j}} \\
\equiv & s_1 - s_2
\end{split} \label{eq:tau0eval}
\end{equation}
and now compare $s_2, s_1$ with $S_2$, $S_1$ in \eqref{eq:sumforms}.  To deal with $s_2$, $S_2$ let us notice that
\begin{equation}
\begin{split}
S_2 - s_2 & = \sum_{j = p}^{\infty} \left[\frac{1}{j + \frac{1}{2}} \cdot \frac{1}{\sqrt{p + j + 1}} - \frac{1}{j} \, \frac{1}{\sqrt{p + j}} \right] \\
& = \sum_{j = p}^{\infty} \left( \left[ \frac{1}{j + \frac{1}{2}} - \frac{1}{j} \right] \frac{1}{\sqrt{p + j + 1}}  + \frac{1}{j} \left[\frac{1}{\sqrt{p + j + 1}} - \frac{1}{\sqrt{p + j}} \right] \right) 
\end{split} \label{eq:s2diffseval}
\end{equation}
so
\begin{equation}
\vertpair{S_2 - s_2} \leq \sum_{j = p}^{\infty} \left( \frac{1}{j^{5/2}} + \frac{1}{j^{5/2}} \right) \leq \frac{2}{p^{3/2}} \label{eq:s2diffsbd}
\end{equation}
With
\begin{equation}
S_1 = \sum_{j = 0}^{p - 1} \frac{1}{j + \frac{1}{2}} \left( \frac{1}{\sqrt{p - j}} - \frac{1}{\sqrt{p + j + 1}} \right) \label{eq:s1renew}
\end{equation}
and $s_1$ in \eqref{eq:tau0eval} we have
\begin{equation}
\begin{aligned}
S_1 - s_1 & =  \sum_{j = 1}^{p - 1} && \left[ \frac{1}{j + \frac{1}{2}} \left( \frac{1}{\sqrt{p - j}} - \frac{1}{\sqrt{p + j + 1}} \right) \right.\\
&&& \left. + \frac{1}{j} \left( - \frac{1}{\sqrt{p + j + 1}} + \frac{1}{\sqrt{p + j}} \right) \right] = \\
& = \sum_{j = 1}^{p - 1} && \left[ \frac{-1}{j} \cdot \frac{1}{\sqrt{p - j} \sqrt{p  + j + 1}} \cdot \frac{1}{\sqrt{p - j} + \sqrt{p  + j + 1}} \right.\\
&&& \left. + \frac{1}{j} \frac{1}{\sqrt{p + j} \sqrt{p + j + 1} ( \sqrt{p + j} + \sqrt{p  + j + 1})} \right]
\end{aligned} \label{eq:s1diffeval}
\end{equation}
and
\begin{equation}
\begin{split}
\vertpair{S_1 - s_1} &\leq \sum_{j = 1}^{p - 1} \left[ \frac{1}{P^{1/2}} \cdot \frac{1}{j(p - j)} + \frac{1}{P^{1/2}} \cdot \frac{1}{j} \cdot \frac{1}{(p + j)} \right]\\
& \leq \frac{2 \log P}{p^{3/2}}.
\end{split} \label{eq:s1diffsbd}
\end{equation}
Inequalities~\eqref{eq:s2diffsbd} and \eqref{eq:s1diffsbd} prove \eqref{eq:tauomega0diff}, i.e., Part (a) is done.
\subsection{Part (b)} \label{subsect:tauomega1}
As in Section~\ref{sec:etabdpt1} we write, $r = 4b$, 
\begin{align}
\begin{aligned}
\tau_1(p) & = &&\sum_{m = 1}^{p - 1} \frac{1}{j} \left[ \frac{ \cos r \sqrt{p - j}}{\sqrt{p - j}} - \frac{\cos r \sqrt{ p + j}}{\sqrt{p + j}} \right] \\
&&&- \sum_{j - p}^{\infty} \frac{1}{j} \frac{\cos r \sqrt{p + j}}{\sqrt{p + j}}
\end{aligned} \label{eq:tau1eval1}\\
\begin{aligned}
& = &&t_1 - t_2
\end{aligned} \label{eq:tau1eval2}
\end{align}
and compare $t_1$, $t_2$ with $S_1, S_2 \in$ \eqref{eq:es1def}, \eqref{eq:es2def} separately.  

Again, an ``easy'' case is to explain
\begin{equation}
\vertpair{S_2 - t_2} \leq C \frac{\log p}{p} \label{eq:es2diffbd}
\end{equation}
Put
\begin{equation}
\mu(x) = \frac{\cos r \sqrt{p + x}}{\sqrt{p + x}}; \label{eq:mu6def}
\end{equation}
then
\begin{equation}
- \mu^{\prime}(x) = \frac{r \sin r \sqrt{p + x}}{p + x} + \frac{\cos r \sqrt{p + x}}{(p + x)^{3/2}} \label{eq:mu6p}
\end{equation} 
and
\begin{equation}
\begin{split}
S_2 - t_2 = \sum_{j = p}^{\infty} & \left[ \left( \frac{1}{j + \frac{1}{2}} - \frac{1}{j} \right) \frac{\cos r \sqrt{p + j + 1}}{\sqrt{p + j + 1}} \right.\\
& + \left. \frac{1}{j} \left( \frac{\cos r \sqrt{p + j + 1}}{\sqrt{p + j + 1}}  - \frac{\cos r \sqrt{p + j}}{\sqrt{p + j}}\right) \right].
\end{split} \label{eq:e2sdiffact}
\end{equation}
The last difference is equal to $\mu^{\prime}(j + \theta_j)$, $0 \leq \theta_j \leq 1$, so 
\begin{align}
\begin{aligned}
\vertpair{S_2 - t_2} & = \sum_{j = p}^{\infty} \left[ \frac{1}{\sqrt{p}} \cdot \frac{1}{j(2j + 1)} + \frac{1}{j} \, \frac{r}{(p + j)} + \frac{1}{j(p + j)^{3/2}} \right] \leq \\
& \leq \left( \frac{1}{2} \frac{1}{p^{3/2}} + \frac{r \log P}{P} + \frac{1}{P^{3/2}} \right)\\
\end{aligned} \label{eq:es2diffeval1}\\
\begin{aligned}
& \leq (r + 1) \frac{\log p}{p}
\end{aligned} \label{eq:es2diffeval2}
\end{align}
\eqref{eq:es2diffbd} is proven.  Finally, with 
\begin{align}
S_1 & = \sum_{j = 1}^{p - 1} \frac{1}{j + \frac{1}{2}} \left[ \frac{\cos r \sqrt{p - j}}{\sqrt{p - j}} - \frac{\cos r \sqrt{p + j + 1}}{\sqrt{p + j + 1}}\right] \label{eq:es1repeat}
\intertext{and}
t_1 & = \sum_{j = 1}^{p - 1} \frac{1}{j} \left[ \frac{\cos r \sqrt{p - j}}{\sqrt{p - j}} - \frac{\cos r \sqrt{p + j}}{\sqrt{p + j}}\right] \label{eq:t1repeat}
\end{align}
\begin{equation}
\begin{split}
S_1 - t_1 = \sum_{j = 1}^{p - 1} &\left\lbrace \parenpair{\frac{1}{j + \frac{1}{2}}  - \frac{1}{j}} \left[ \frac{\cos r \sqrt{p - j}}{\sqrt{p - j}} - \frac{\cos r \sqrt{p + j + 1}}{\sqrt{p + j + 1}}\right] \right.\\
& \left. + \frac{1}{j} \left( \frac{\cos r \sqrt{p + j}}{\sqrt{p + j}} - \frac{\cos r \sqrt{p + j + 1}}{\sqrt{p + j + 1}}  \right) \right\rbrace  = 
\end{split} \label{eq:s1t1diffeval1}
\end{equation}
\begin{equation}
\begin{split}
= \sum_{j = 1}^{p - 1} \bracepair{ \frac{1}{j(2j + 1)} \mu^{\prime}(\gamma_j)(2j + 1) + \frac{1}{j} \mu(j + \delta_j) } \quad \text{where}
\end{split} \label{eq:s1t1diffeval2}
\end{equation}
\begin{equation}
-j \leq \gamma_j \leq j + 1, \quad 0 \leq \delta_j \leq 1 \label{eq:inbetweenconsts}
\end{equation}
so
\begin{equation}
\begin{split}
\vertpair{S_1 - t_1} \leq \sum_{j = 1}^{p - 1} & \left[ \frac{1}{j} \left( \frac{r}{2(p - j)} + \frac{1}{2(p-j)^{3/2}} \right) \right.\\
& \left. \frac{1}{j} \left( \frac{r}{2(p + j)} + \frac{1}{2(p + j)^{3/2}} \right) \right] \leq \\
\leq \sum_{j = 1}^{p - 1} & \frac{1}{j} \, \frac{r + 1}{2} \left[ \frac{1}{p - j} + \frac{1}{p + j} \right] \leq 2 (r + 1) \, \frac{\log P}{P}
\end{split} \label{eq:s1t1diffact}
\end{equation}
This inequality, together with \eqref{eq:es2diffeval2}, prove that
\begin{equation}
\begin{split}
\vertpair{\tau_1(p) - \omega_1(p)} & \leq \vertpair{S_1 - t_1} + \vertpair{S_2 - t_2}\\
& \leq 3 (r + 1) \frac{\log P}{P}
\end{split} \label{eq:tau1diffeval}
\end{equation}
and \eqref{eq:tauomega1diff}.  

Lemma~\ref{lem:tauomegadiff} is proven.
\end{proof}%Lemma VI.12
This lemma shows that up to the error-term $\disp O \left( \frac{\log P}{P} \right)$ the evaluation of $\evencasesig(n) \in$ \eqref{eq:sigmaprepeat} gives the same result for the potential $v^e$ as $\oddcasesig(n) \in$ \eqref{eq:finalsigmatwiddle} for the potential $v^o$, at least, if $n$ is even.  [The adjustment as in Section~\ref{subsect:oddcomplete} shows that we get the same result for $n$ odd.  We omit details.]  
\subsection{}
\begin{cor} \label{cor:et2n}
If $v^e \in $ \eqref{eq:pointeven} then 
\begin{equation}
\begin{split}
T_2(n) & = 2 a_n^2 \evencasesig(n) = \\
& = \frac{\mysteryletterB(n)}{n} + O \parenpair{\frac{\log n}{n^{3/2}}}
\end{split} \label{eq:t2evenbd}
\end{equation}
where
\begin{align}
\mysteryletterB(n) = \frac{1}{\mysteryconstA} \left[ (-1)^{n + 1} \sin 2b \sqrt{2n} - \frac{1}{2} \sin 4b \sqrt{2n} \right] \label{eq:mysterysymb}
\end{align}
\end{cor}

\begin{proof}
We showed in previous subsections \zref[subsection]{subsect:tauomega0},\zref[subsection]{subsect:tauomega1} that in the case of an even potential $v^e$ $T_2(n)$ behave in the same way as in the case $v^o$ so 
\begin{equation}
T_2(n) = 2 a_n^2 \evencasesig(n)
\end{equation} 
where by \eqref{eq:anothersigmaodd}
\begin{align}
\evencasesig(n) &= \frac{1}{2} (-1)^{n + 1} \frac{\sin 2b \sqrt{2n}}{\sqrt{2n}} + O \parenpair{ \frac{\log n}{n} }. \label{eq:anothersigmaeven}\\
\intertext{Recall that}
2 a_n^2 &= \frac{2}{\pi} \frac{1}{\sqrt{2n}} \left[ 1 + (-1)^n \cos (2 b \sqrt{2n}) \right] + O \parenpair{\frac{1}{n}}.  \label{eq:anotheranform}
\end{align}
or more precisely, see \eqref{eq:a2eeval}.  These identities imply \labelcref{eq:t2evenbd,eq:mysterysymb}.
\end{proof}

\subsection{}  Although for an odd perturbation $v^o$
\begin{equation}
T_3(n) = 0, \quad n \geq N_*,
\end{equation}
this is not true for an even $v^e$; i.e., $T_3(n; v^e)$ does hardly vanish.  

By \eqref{eq:t3even}

\begin{equation}
T_3(n) = 2 a_n^2 \evencasesig(n)^2 + 2 a_n^4 \tau^{\prime}(n), \label{eq:et3even} 
\end{equation}
where
\begin{equation}
\tau^{\prime}(n) = \sideset{}{^{\prime}}\sum_{\substack{k = 0 \\ k - n \text{ even}}}^{\infty} \frac{a_k^2}{(n - k)^2}. \label{eq:tau3pdef}
\end{equation}
Notice that
\begin{equation}
\begin{split}
\sum_{\substack{k = 1 \\ k - n \text{ even}}}^{\infty} \frac{1}{\sqrt{k}} \, \frac{1}{(n - k)^2} &= \sum_{i = 1}^{n - 1} \frac{1}{i^2} \left[ \frac{1}{\sqrt{n-i}} + \frac{1}{\sqrt{n+i}} \right] + O \parenpair{\frac{1}{n^{3/2}}}\\
& = O \parenpair{\frac{1}{\sqrt{n}}}
\end{split} \label{eq:tau3peval1} 
\end{equation}
It implies (with an analogue of Lemma~\ref{lem:tauomegadiff} that 
\begin{equation}
T_3(n) = \orderasymp{\frac{1}{n^{3/2}}} . \label{eq:t3leading}
\end{equation}
Without specifics of formulas \eqref{eq:et3even} and \eqref{eq:anothersigmaeven}, \eqref{eq:anotheranform} we had a general claim [see Corollary~\ref{cor:tjnorm}, \eqref{eq:tjnorm}, $\disp \alpha = \frac{1}{4}$, $j = 3$, and Corollary~\ref{cor:lamest}, $q = 2$ and $3$.]:
\begin{equation}
\vertpair{T_3(n)} \leq C \left( \frac{\log n}{\sqrt{n}} \right)^3 \label{eq:t3earlyest}
\end{equation}
so \eqref{eq:t3leading} brings a slight gain of a factor $\left( \log n \right)^3$.  

But $T_4$ and the entire tail 
\begin{equation}
\sum_{j = 4}^{\infty} \vertpair{T_j(n)} \leq \left( C \frac{\log n}{\sqrt{n}} \right)^4 = \orderasymp{\frac{1}{n^{3/2}}}
\end{equation}
These observations, in particular \eqref{eq:t3leading}, show that in both cases $v^o$ and $v^e$ we have by \eqref{eq:lamest} of Corollary~\ref{cor:lamest}:
\begin{equation}
\lambda_n = (2n + 1) + T_1(n) + T_2(n) + \orderasymp{\frac{\log n}{n^{3/2}}}
\end{equation}
Explicit form of $T_2(n)$ for odd $v^o$ is incorporated in \eqref{eq:lamasymp} of Theorem~\ref{thm:eigendistr}.  Now we are ready to give its analog for even $v^e$.  

\begin{thm} \label{thm:evenasymp} %VI.24
In the case of the potential perturbation $v^e(x) = \pointmass{x - b} + \pointmass{x + b}$ asymptotically 
\begin{equation}
\lambda_n  = (2n + 1) + \frac{1}{\sqrt{2n}} \chi(n) + \frac{(-1)^{n+1}}{n} \left( \zeta(n) + \omega(n) \right) + \orderasymp{\frac{\log n}{n^{3/2}}} , \label{eq:lameasymp}
\end{equation} 
where
\begin{align}
\chi(n) & = \frac{4}{\pi} \left[ 1 + (-1)^n \cos 2b \sqrt{2n} \right], \label{eq:chidef}\\
\zeta(n) & = \frac{b}{\pi} \left( 1 - \frac{b^3}{3} \right)\sin 2b \sqrt{2n} \label{eq:zetatruedef}\\
\omega(n) & = \frac{1}{8} + (-1)^n \, \frac{1}{2} \sin 4 b \sqrt{2n} \label{eq:omegatruedef}
\end{align}
\end{thm}

\subsection{}  Of course, if 
\begin{equation}
W(x) = tv^e(x) \label{eq:wscaleddef} 
\end{equation}
then
\begin{equation}
T_j(n; W) = t^j T_j(n; v^e) \, , \label{eq:tjscaled}
\end{equation}
and
\begin{equation}
\begin{split}
\lambda_n(W) = (2n + 1) + t T_1(n; v^e) + t^2 T_2(n; v^e) + \orderasymp{\frac{\log n}{n^{3/2}}}
\end{split} \label{eq:lamformscaled} 
\end{equation}
and the term ``of order $\disp \frac{1}{n}$'' in $T_1$ comes with the coefficient $t$ although the term\\ ``of order $\disp \frac{1}{n}$'' in $T_2$ comes with the coefficient $t^2$.  This observation implies the following.  
\begin{cor}
If the potential is \eqref{eq:wscaleddef} then 
\begin{equation}
\lambda_n  = (2n + 1) + \frac{t}{\sqrt{2n}} \chi(n) + \frac{1}{n} \left[t \zeta(n) + t^2\omega(n) \right] + \orderasymp{\frac{\log n}{n^{3/2}}} , \label{eq:lamescaledasymp}
\end{equation}
with $\chi$, $\zeta$, $\omega$ defined in \labelcref{eq:chidef,eq:zetatruedef,eq:omegatruedef}. 
\end{cor}
Finally, we are ready to get asymptotic of eigenvalues in the case of any two-point $\bracepair{\pm b}$ $\delta$-potential.  
\begin{thm} \label{thm:gen2ptpot} %VI.27
Let 
\begin{equation}
w(x) = tv^e(x) + sv^o(x), \quad t, s \in \CC. \label{eq:gen2ptpot}
\end{equation}
Then 
\begin{equation}
L = - \frac{d^2}{dx^2} + x^2 + w(x) \label{eq:elllast}
\end{equation}
has a discrete spectrum, all its eigenvalues $\bracepair{\lambda_n}$ in the half-plane 
\[
H(N_*) = \bracepair{z: \Re z \geq 2 N_*} 
\]
are simple, and 
\begin{equation}
\begin{split}
\lambda_n = (2n + 1) + \frac{t}{\sqrt{2n}} \chi(n) + \frac{1}{n} \left[ t \zeta(n) + (t^2 + s^2) \omega(n) \right] + \orderasymp{\frac{\log n}{n^{3/2}}} 
\end{split} \label{eq:lamlastdecomp}
\end{equation}
\end{thm}
\begin{proof}
$T_1$ is a linear function on $w$'s, so 
\begin{equation}
\begin{split}
T_1(n; W) &= t T_1(n; v^e) + s T_1(n; v^e) \\
& = t T_1(n; v^e).
\end{split} \label{eq:t1.2pt}
\end{equation}
This explains two terms with $t$-coefficient in \eqref{eq:lamlastdecomp}.  Next,
\begin{equation}
\begin{split}
T_2(n; W) &= \tra \frac{1}{2\pi i} \int\limits_{\bd{\mathcal{D}_n}} (z - z_n) R^0 W R^0 W R^0 \, dz\\
&= t^2 T_2(n; v^e) + s^2 T_2(n; v^0) + tsT^{\prime \prime} \quad \text{where}
\end{split} \label{eq:t2.2pt} 
\end{equation}
\begin{equation}
\begin{split}
T^{\prime \prime} = \tra \frac{1}{2\pi i} \int\limits_{\bd{\mathcal{D}_n}} (z - z_n) \left[ R^0 v^e R^0 v^o R^0 + R^0 v^o R^0 v^e R^0 \right] \, dz = 0.
\end{split} \label{eq:t2.2pt.b}
\end{equation}
Indeed 
\begin{equation}
\begin{split}
v^e_{mk} & = 0 \quad \text{if } m - k \text{ odd}\\
\text{and} \quad v_o^{kj} & = 0 \quad \text{if } k - j \text{ even}
\end{split} \label{eq:vdecomp}
\end{equation}
So in \eqref{eq:t2.2pt.b}
\[
v^e_{mk} v^o_{km} = 0 \quad \text{and} \quad v^o_{mj}v^e_{jm} = 0
\]
for any $\bracepair{m, k}$ or $\bracepair{m, j}$ because at least one factor is zero.  
\end{proof}
\section{Comments; miscellaneous} \label{sec:comment}  %Section VII!
The technical estimates of Sections~\ref{sec:manyineqs} to \ref{sec:etabdpt2} [mostly Section ~\ref{sec:etabdpt1}] lead to asymptotics of the eigenvalues in the case of perturbations \eqref{eq:wtype2}, at least, if the number of point interactions is finite, i.e., 
\begin{equation}
w(x) = \sum_{j = 1}^J c_j \pointmass{x - b_j}, \quad J < \infty. \label{eq:pointmassrenew}
\end{equation}
Of course, if $K = 2$ and $b_2 + b_1 = 0$ the interactions are well balanced and many terms in the asymptotic formulas vanish.  But for any set of points $(b_k)$ and coupling coefficients $(c_k)$ we could give general formulas if we know the values of $\bracepair{T_j(n; w)}$.  Let us do this exercise for $j = 1, \, 2$.  
\subsection{Case \texorpdfstring{$b = 0$, $c = 2$}{b = 0, c = 2}; i.e., \texorpdfstring{$v_*(x) = 2 \pointmass{x}$}{v(x) = 2 delta(x)}}  We put $c = 2$ to make possible a direct comparison with the case of perturbations $v^e$ in \eqref{eq:pointeven}.  It seems we could use Thm.~\ref{thm:evenasymp} and its formulas to claim Prop.~\ref{prop:singleasymp} (see below).  But in estimates of Sections \ref{sec:manyineqs}--\ref{sec:etabdpt2} we often used an assumption $b \neq 0$; see for example, inequalities \eqref{eq:rho11bd}, \eqref{eq:es2absbd}; \eqref{eq:rho17bd}, \eqref{eq:q1q1pdiff}; \eqref{eq:q2q2pdiff}, \eqref{eq:rho25bd} which are important in evaluation of the error-term \eqref{eq:rhobd} or its analogues.  However, if $b = 0$ we do not face such terms and evaluation is straightforward.

Indeed, by \eqref{eq:hefcnform} if $w = v_*$
\begin{align}
a_n &= h_n(0) = \frac{2^{1/4}}{\pi^{1/2}} \cdot \frac{1}{n^{1/4}} \cdot \cos \parenpair{n \, \frac{\pi}{2}} + \orderasymp{\frac{1}{n^{5/4}}} \label{eq:an0est} \\
\intertext{and}
a_k^2 & = \begin{cases} \displaystyle \frac{2}{\pi} \cdot \frac{1}{\sqrt{2k}} + \orderasymp{\frac{1}{k^{3/2}}}, & k > 0 \text{ even} \\
0, & k \text{ odd.}\end{cases} \label{eq:an02estmain}
\end{align}
Therefore, if $n \geq N_*$,
\begin{equation} \label{eq:t10decomp}
T_1(n; v_*) = 2 a_n^2 = \begin{cases} \displaystyle \frac{4}{\pi} \cdot \frac{1}{\sqrt{2n}} + \rho_{30}(n), & n \text{ even} \\
0, & n \text{ odd}.  \end{cases}
\end{equation}
where
\begin{equation} \label{eq:rho30bd}
\vertpair{\rho_{30}(n)} \leq \frac{C}{n^{3/2}}.
\end{equation}
Next term by \eqref{eq:t2resolve} is
\begin{equation} \label{eq:t20decomp}
T_2(n) = 2 a_n^2 \cdot \widetilde{\sigma}(n), 
\end{equation}
where
\begin{equation} \label{eq:varsigmarednew}
\widetilde{\sigma}(n) =  \sideset{}{'}\sum_{k = 0}^{\infty} \frac{a_k^2}{n - k}.
\end{equation}
The subsum $\widetilde{\sigma}_{\text{odd}}$ by \eqref{eq:t10decomp} and Remark~\ref{rem:sbds}, $\displaystyle \beta = \frac{3}{2} > 1$,
\begin{equation} \label{eq:sigrenewoddbd}
\widetilde{\sigma}_{\text{odd}} = \orderasymp{\frac{1}{n}}, 
\end{equation}
and with the same remarks
\begin{equation} \label{eq:sigrenewevendecomp}
\widetilde{\sigma}_{\text{even}} = \frac{2}{\pi} \sum_{\substack{k = 1 \\ k \text{ even} \\ k \neq n}}^{\infty} \frac{1}{\sqrt{2k}} \cdot \frac{1}{n - k} + \orderasymp{\frac{1}{n}}.
\end{equation}
Now we refer to Lemma~\ref{lem:zetabds}, \eqref{eq:zetabd}, and Lemma~\ref{lem:tauomegadiff}, \eqref{eq:tauomega0diff}, to claim
\begin{equation}\label{eq:sigrenewevenbd}
\widetilde{\sigma}_{\text{even}} = \orderasymp{\frac{1}{n}}.
\end{equation}
This leads us to the following.
\begin{lem} \label{lem:t10complete} If $w(x) = 2 \pointmass{x}$
\begin{enumerate}[label = (\roman*)]
\item \begin{equation} \label{eq:t10decomprepeat}
T_1(n; v_*) = \begin{cases} \displaystyle \frac{4}{\pi} \cdot \frac{1}{\sqrt{2n}} + \rho_{30}(n) , & n \text{ even}, \\
0, & n \text{ odd}.  \end{cases}
\end{equation}
where
\begin{equation} \label{eq:rho30bdrepeat}
\vertpair{\rho_{30}(n)} \leq \frac{C}{n^{3/2}},
\end{equation}
and
\item 
\begin{equation} \label{eq:t20bd}
T_2(n) = \orderasymp{\frac{1}{n^{3/2}}}
\end{equation}
\end{enumerate}
\end{lem}
\begin{proof}
Part (i) has been explained by formulas \eqref{eq:an0est} -- \eqref{eq:rho30bd}.  With $a_n^2$ given by \eqref{eq:t10decomp}
\begin{equation}
a_n^2 = \orderasymp{\frac{1}{\sqrt{n}}},
\end{equation}
and 
\begin{equation}
\widetilde{\sigma} = \orderasymp{\frac{1}{n}}
\end{equation}
by \eqref{eq:sigrenewoddbd} and \eqref{eq:sigrenewevenbd}.  Therefore, their product \eqref{eq:t20decomp} is $\displaystyle \orderasymp{\frac{1}{n^{3/2}}}$, i.e., \eqref{eq:t20bd} holds.
\end{proof}
Finally, the term $T_3(n)$ --- see \eqref{eq:tau3pdef},\eqref{eq:et3even} or \eqref{eq:t3even} --- is by \eqref{eq:tau3pdef} -- \eqref{eq:t3leading} and the formulas of the previous subsection is  $\disp \leq C \left[ \frac{1}{\sqrt{n}} \cdot \parenpair{\frac{1}{n}}^2 + \frac{1}{n} \cdot \frac{1}{\sqrt{n}} \right]$
so
\begin{equation} \label{eq:t30bd}
\vertpair{T_3(n)} \leq \frac{C}{n^{3/2}}.
\end{equation}
We have proven the following.
\begin{prop} \label{prop:singleasymp}
If $w = tv_* = 2t\pointmass{x}$ the eigenvalues of the operator \eqref{eq:elldecompintro} for $n \geq N_*$ are 
\begin{align}
\lambda_n & =   (2n + 1) + \frac{4t}{\pi} \cdot \frac{1}{\sqrt{2n}} + \orderasymp{\frac{1}{n^{3/2}}}, & n \text{ even}\\
& =  (2n + 1) + \orderasymp{\frac{1}{n^{3/2}}}, & n \text{ odd}.
\end{align}
\end{prop}
\subsection{}  Let us conclude our example with a simple case of one-point interaction
\begin{equation} \label{eq:w1pt}
w(x) = t \pointmass{x - b}, \quad b \neq 0.
\end{equation}
Again recall that 
\begin{equation}\label{eq:hermfcnformrenew}
\begin{split}
h_k(x) = \frac{2^{1/4}}{\pi^{1/2}} \cdot \frac{1}{k^{1/4}} &\left[ \cos \left( x \sqrt{2k + 1} - k \frac{\pi}{2} \right)  + \right.\\
 &+ \left. \frac{x^3}{6} \frac{1}{\sqrt{2k + 1}} \sin \left( x \sqrt{2k + 1} - k \frac{\pi}{2} \right) + \orderasymp{\frac{1}{k}} \right] 
\end{split}
\end{equation}
and if $x = b$
\begin{equation}\label{eq:ak2fcnformrenew}
\begin{aligned}
h_k(x) &= \frac{2^{1/2}}{\pi} \cdot \frac{1}{\sqrt{k}} \cdot \frac{1}{2} \cdot &&\left[1 + (-1)^k \cos \left(2 b \sqrt{2k + 1} \right)  + \right.\\
& &&+ \left. \frac{b^3}{6} \frac{(-1)^k}{\sqrt{2k + 1}} \sin \left(2 b \sqrt{2k + 1} \right) + \orderasymp{\frac{1}{k}} \right] \\
&= \frac{1}{\pi} \cdot \frac{1}{\sqrt{2k}} \cdot  &&\left[1 + (-1)^k \cos \left(2 b \sqrt{2k} \right)  - \right.\\
& &&- \left. \frac{b^3}{6} \frac{(-1)^k}{\sqrt{2k}} \cdot b\left( 1 - \frac{b^2}{3} \right) \sin \left(2 b \sqrt{2k } \right) + \orderasymp{\frac{1}{k}} \right]
\end{aligned}
\end{equation}
Then
\begin{align} 
T_1(n) & = W_{nn} = t a_n^2 \label{eq:w1t1} \\
\intertext{and}
T_2(n) & = \frac{t^2}{2} a_n^2 \sum_{k = } \frac{a_k^2}{n - k} \label{eq:w1t2} 
\end{align}
Four terms in brackets on the right side of \eqref{eq:ak2fcnformrenew} give impacts into 
\begin{equation}\label{eq:sigmabardef}
\overline{\sigma}(n) = \sideset{}{^{\prime}} \sum_{k = 0}^{\infty} \frac{a_k^2}{n - k} 
\end{equation}
which are evaluated as
\begin{equation}\label{eq:w1basicsum}
\sideset{}{^{\prime}} \sum_{k = 1}^{\infty} \frac{1}{\sqrt{k}} \cdot \frac{1}{n - k} = \orderasymp{\frac{\log n}{n}} \quad \text{by Lemma~\ref{lem:zetabds}, \eqref{eq:zetabd},}
\end{equation}
\begin{equation} \label{eq:w1basicsum2}
\begin{split}
\sideset{}{^{\prime}}\sum_{k = 1}^{\infty} \frac{(-1)^k}{\sqrt{k}} \cdot \frac{\cos 2b \sqrt{2k}}{n - k} & = (-1)^n \left[ \, \sideset{}{^{\prime}}\sum_{\substack{k = 1 \\ k - n \text{ even}}}^{\infty} - \sum_{\substack{k = 1 \\ k - n \text{ odd}}}^{\infty} \, \right]\\
& = (-1)^n \left[c_* \frac{\sin 2b \sqrt{2n}}{\sqrt{2n}} - c_* \frac{\sin 2 b \sqrt{2n}}{\sqrt{2n}} + \orderasymp{\frac{\log n}{n}}\right] \\
& = \orderasymp{\frac{\log n}{n}}
\end{split} 
\end{equation}
by Prop~\ref{prop:es2bds}, \eqref{eq:etarepeat} or \eqref{eq:es2bd}. [It is important that coefficient $c_*$ is the same for subsums over even and odd $k$.]

The third and fourth terms are evaluated just by absolute value so
\begin{equation}\label{eq:thirdterm}
\sideset{}{^{\prime}} \sum_{k = 1}^{\infty} \frac{1}{\sqrt{k}} \cdot \frac{1}{\sqrt{k}} \cdot \frac{1}{\vert n - k \vert} = \orderasymp{\frac{\log n}{n}} \text{ by Remark~\ref{rem:sbds}, }\beta = 1,
\end{equation}
and
\begin{equation}\label{eq:fourthterm}
\sideset{}{^{\prime}} \sum_{k = 1}^{\infty} \frac{1}{\sqrt{k}} \cdot \frac{1}{k} \cdot \frac{1}{\vert n - k \vert} = \orderasymp{\frac{1}{n}} \text{ by Remark~\ref{rem:sbds}, }\beta = \frac{3}{2}.
\end{equation}
The estimates \eqref{eq:w1basicsum} -- \eqref{eq:fourthterm} explain that
\begin{equation} \label{eq:sigmabarbd}
\overline{\sigma}(n) = \orderasymp{\frac{\log n}{n}}
\end{equation}
so by \eqref{eq:w1t2} and \eqref{eq:ak2fcnformrenew}
\begin{equation}\label{eq:w1t2bd}
T_2(n) = \orderasymp{\frac{\log n}{n^{3/2}}}
\end{equation}
We have proven the following
\begin{prop}\label{prop:w1lamform}
If $w = t \pointmass{x - b}$ then the eigenvalues of the operator \eqref{eq:elldecompintro} for $n \geq N_*$ are
\begin{equation}\label{eq:w1lamform}
\begin{split}
\lambda_n = (2n + 1) &+ \frac{1}{\pi \sqrt{2n}} \left[ 1 + (-1)^n \cos (2b \sqrt{2n}) \right] \\
& - \frac{1}{\pi} \frac{(-1)^n}{2n} b \left( 1 - \frac{b^2}{3} \right) \sin (2b \sqrt{2n}) + \orderasymp{\frac{\log n}{n^{3/2}}}.
\end{split} 
\end{equation}
\end{prop}

\subsection{} \label{subsect:extralog} One more technical adjustment should be to cover the case $p = 2$ in Proposition~\ref{prop:imagw}.  If the condition \eqref{eq:wcond} would be changed to 
\begin{equation} \label{eq:wcondadjust}
\vert w_{jk} \vert \leq C_0 \cdot \frac{\log(e + j)}{(1 + j)^{\alpha}} \cdot \frac{\log(e + k)}{(1 + k)^{\alpha}}
\end{equation}
we could follow the lines of Section~\ref{sec:prelims} but instead of \eqref{eq:nuexpr}, \eqref{eq:mudef} we have to evaluate 
\begin{equation} \label{eq:muadjust}
\mu = \sum_{j = 0}^{\infty} \frac{\log(e + j)}{(1 + j)^{2\alpha}} \cdot \frac{1}{\vertpair{z - z_j}}, \quad z \in \bd{\mathcal{D}_n}
\end{equation}
We omit technicalities which mimick steps \eqref{eq:smallvarianceterms} -- \eqref{eq:malphdef}; they certainl lead us to the estimates (compare \eqref{eq:muest}, \eqref{eq:malphdef} )
\begin{equation} \label{eq:muadjustest}
\mu \leq \frac{G(\alpha)}{n^{2\alpha}} \left[ \log (en) \right]^2, \quad 2\alpha < 1.
\end{equation}
If $\disp \alpha = \frac{1}{8}$ by \eqref{eq:somebd}
\begin{equation}\label{eq:wcondbetter}
\vert w_{jk} \vert \leq C_0 \frac{\log (ej)}{j^{1/8}} \cdot \frac{\log(ek)}{k^{1/8}}
\end{equation}
and 
\[
\mu \leq \frac{G^*}{n^{1/4}} \left[ \log (en) \right]^2 .
\]
Therefore, in Proposition~\ref{prop:imagw}, $p = 2$, $N^*$ should be chosen to guarantee [compare \eqref{eq:somebd}]
\begin{gather}
C_0 G^* \nu \frac{(\log en)^2}{n^{1/4}} \leq \frac{1}{2}, \quad \text{or}\\
\frac{c}{2} \nu^{1/2} \frac{\log en}{n^{1/8}} \leq \frac{1}{2}, \quad c^2 = (2 C_0 G^*).
\end{gather}
By \eqref{eq:tfromndef}, $\beta = \frac{1}{8}$, and Corollary~\ref{cor:xtbehave} $N^*$ could be chosen as
\begin{align*}
X_{1/8} \left( \frac{c}{2} \nu^{1/2} \right) &\leq 2^8 \left( \frac{c}{2} \nu^{1/2} \log \left( 8A \nu^{1/2} \frac{c}{2} \right) \right)^8\\
& \leq C_2 \left( \nu( \log e\nu)^2 \right)^4, \, \exists C_2.
\end{align*}
This explains the claim \eqref{eq:newnstar.pmid} in Proposition~\ref{prop:imagw}.

\section*{Acknowledgements}
The author is indebted to Charles Baker and Petr Siegl for numerous discussions.  Without their support this work would hardly be written, at least in a reasonable period of time.  I am also thankful to Daniel Elton, Paul Nevai, and G\"{u}nter Wunner for valuable comments and information related to topics of this manuscript.  
\printbibliography

\end{document}